\documentclass[12pt, reqno]{amsart}
\usepackage{amssymb, amsmath, amsthm,mathrsfs,calligra}
\usepackage[backref]{hyperref}
\usepackage[alphabetic,backrefs,lite]{amsrefs}
\usepackage{amscd}   
\usepackage[twoside,height=9in,width=6.6in]{geometry}
\usepackage[all]{xy} 
\usepackage{setspace}
\usepackage{mathtools}
\usepackage{marginnote}
\usepackage[normalem]{ulem} 
\DeclarePairedDelimiter{\floor}{\lfloor}{\rfloor}
\DeclarePairedDelimiter{\ceil}{\lceil}{\rceil}

\usepackage{tikz}
\usepackage{multicol}

\DeclareFontEncoding{OT2}{}{} 

\newtheorem{lemma}{Lemma}[section]
\newtheorem{theorem}[lemma]{Theorem}
\newtheorem{proposition}[lemma]{Proposition}
\newtheorem{prop}[lemma]{Proposition}
\newtheorem{cor}[lemma]{Corollary}

\newtheorem{claim*}{Claim}
\newtheorem{thm}[lemma]{Theorem}

\theoremstyle{definition}
\newtheorem{remark}[lemma]{Remark}

\newtheorem{rmk}[lemma]{Remark}
\newtheorem{example}[lemma]{Example}

\newcommand{\Aff}{{\mathbb A}}
\renewcommand{\AA}{\Aff}

\newcommand{\PP}{{\mathbb P}}
\newcommand{\C}{{\mathbb C}}
\newcommand{\F}{{\mathbb F}}
\newcommand{\Q}{{\mathbb Q}}

\newcommand{\Z}{{\mathbb Z}}

\newcommand{\sE}{{\mathscr E}}
\newcommand{\sL}{{\mathscr L}}
\newcommand{\sF}{{\mathscr F}}
\newcommand{\sO}{{\mathscr O}}
\DeclareMathOperator{\sHom}{\mathscr{H}\text{\kern -4pt {\calligra\large om}}\,}

\newcommand{\calA}{{\mathcal A}}
\newcommand{\calB}{{\mathcal B}}

\newcommand{\calE}{{\mathcal E}}

\newcommand{\calFA}{{\mathcal F}}

\newcommand{\calL}{{\mathcal L}}

\newcommand{\calO}{{\mathcal O}}

\newcommand{\frakm}{{\mathfrak m}}

\DeclareMathOperator{\HH}{H}

\DeclareMathOperator{\Hom}{Hom}
\DeclareMathOperator{\Ext}{Ext}

\DeclareMathOperator{\ord}{ord}

\DeclareMathOperator{\Spec}{Spec}

\DeclareMathOperator{\res}{res}

\newcommand{\Xpc}{X_\mathrm{pc}}
\newcommand{\Xms}{X_\mathrm{ms}}
\newcommand{\Yms}{Y_\mathrm{ms}}

\newcommand{\isom}{\simeq}

\numberwithin{equation}{section}
\numberwithin{table}{section}

\newcommand{\defi}[1]{\emph{#1}} 

\title[Symmetric differentials]{Explicit computation of symmetric differentials and its application to quasi-hyperbolicity}
\hypersetup{
	pdfauthor   = {},
	pdftitle    = {Symmetric differentials},
	pdfsubject  = {},
	pdfkeywords = {},
	backref=true, pagebackref=true, hyperindex=true, colorlinks=true,
	breaklinks=true, urlcolor=blue, linkcolor=blue, citecolor=blue,
	bookmarks=true, bookmarksopen=true}

\author{Nils Bruin}
\author{Jordan Thomas}
\author{Anthony V\'arilly-Alvarado}
\thanks{ The first author acknowledges the support of the Natural Sciences and Engineering Research Council of Canada (NSERC), funding reference number RGPIN-2018-04191. The third author was partially supported by  NSF grants DMS-1352291 and DMS-1902274. }

\address{Department of Mathematics, Simon Fraser University, Burnaby, BC, Canada V5A 1S6}
\email{nbruin@cecm.sfu.ca}
\urladdr{http://www.cecm.sfu.ca/~nbruin}

\address{Department of Mathematics, Courant Institute of Mathematical Sciences, New York University, 251 Mercer Street, New York, NY 10012}
\email{jwthomas04@gmail.com}

\address{Department of Mathematics MS 136, Rice University, 6100 S.\ Main St., Houston, TX 77005, USA}
\email{av15@rice.edu}
\urladdr{http://math.rice.edu/\~{}av15}

\date{July 23, 2021}

\begin{document}

\begin{abstract}
We develop explicit techniques to investigate algebraic quasi-hyperbolicity of singular surfaces through the constraints imposed by symmetric differentials.  We apply these methods to prove that rational curves on Barth's sextic surface, apart from some well-known ones, must pass through at least four singularities, and that genus 1 curves must pass through at least two.  On the surface classifying perfect cuboids, our methods show that rational curves, again apart from some well-known ones, must pass through at least seven singularities, and that genus 1 curves must pass through at least two.

We also improve lower bounds on the dimension of the space of symmetric differentials on surfaces with $A_1$-singularities, and use our work to show that Barth's decic, Sarti's surface, and the surface parametrizing $3\times 3$ magic squares of squares are all algebraically quasi-hyperbolic.

\end{abstract}

\maketitle

\section{Introduction}

A complex projective surface $Y$ is \emph{algebraically hyperbolic} if there is an $\epsilon > 0$ such that for every curve $C\subset Y$ of geometric genus $g(C)$, we have $2g(C) - 2 \geq \epsilon\deg(C)$; in particular, such a surface does not contain curves of geometric genus $0$ or $1$.  In this article we call a complex projective surface $Y$ \emph{algebraically quasi-hyperbolic} if it contains only finitely many curves of geometric genus $0$ or $1$. 

Algebraically quasi-hyperbolic surfaces abound: for example, if $Y/\C$ is a smooth proper surface of general type whose Chern classes satisfy $c_1(Y)^2 > c_2(Y)$, then $Y$ is algebraically quasi-hyperbolic, by work of Bogomolov~\cite[Corollary~5]{Bogomolov}. Surfaces of general type with an ample cotangent bundle satisfy the requisite Chern class inequality~\cite{FultonLazarsfeld}. On the other hand, for a smooth complex surface $Y\subset \PP^3$ of degree $d \geq 5$ the inequality $c_1(Y)^2 > c_2(Y)$ does not hold; nevertheless, genus bounds of Xu~\cite{Xu} show that a very general such surface is also algebraically quasi-hyperbolic, a statement that had been conjectured by Harris.  Recently, Coskun and Riedl~\cite{CoskunRiedl} improved Xu's bounds to show that a very general complex surface $Y\subset \PP^3$ of degree $d \geq 5$ is in fact algebraically hyperbolic\footnote{See~\cite[Remark~3.4]{Demailly} for a history of progress on genus bounds for curves on a very general surface $Y\subset \PP^3$.}.

Simple abelian surfaces are also algebraically quasi-hyperbolic in the sense above: it is well-known that every map from $\PP^1$ to any abelian variety is constant (e.g., such a map necessarily factors through the Albanese variety of $\PP^1$, which is a point). Alternatively, and in the spirit of our work, a regular differential on an abelian surface $Y$ would pull back to the zero differential on any genus $0$ curve $C\subset Y$. Since the cotangent sheaf $\Omega^1_Y$ is globally generated, no such $C$ can exist.

Not all is lost if $Y$ is a surface for which $\Omega^1_Y$ has no global sections, as one can mimic the above argument with a different sheaf, for instance  the symmetric powers $S^m\Omega^1_Y$. This idea is central to Bogomolov's result~\cite{Bogomolov} that if the Chern classes of $Y$ satisfy the inequality $c_1(Y)^2>c_2(Y)$, then $S^m\Omega^1_Y$ has global sections for large enough $m$, and as a consequence $Y$ is algebraically quasi-hyperbolic.

Bogomolov and de Oliveira observed in~\cite{BO} that if $\tau\colon Y\to X$ is a minimal resolution of a surface $X$ with $A_1$ singularities, then it is possible that $S^m\Omega^1_Y$ has global sections for large enough $m$ even when $c_1(Y)^2\leq c_2(Y)$.  

In principle, symmetric differentials on $Y$, or more generally elements of $\HH^0(X,(\tau_*S^m\Omega^1_Y)^{\vee\vee})$, can be used to constrain the locus of genus $0$ or $1$ curves on $X$.  To wit, an element $\omega \in \HH^0(Y,S^m\Omega^1_Y)\isom \HH^0(\PP(\Omega^1_Y),\sO_{\PP(\Omega^1_Y)}(m))$ defines a surface $Y'\subset \PP(\Omega^1_Y)$ in the cotangent bundle of $Y$ that is a multisection of degree $m$ for the projection $\PP(\Omega^1_Y) \to Y$. If, for example, $\Omega^1_Y$ is big then any curve of genus $0$ or $1$ on $Y$ lifts to a leaf of the foliation on $Y'$ induced by $\omega$. This idea is already present in Bogomolov's work~\cites{Bogomolov,Deschamps}, and was amplified by Green and Griffiths in~\cite{GreenGriffiths}. However, computing the integral curves defined by the degree $m$ first order differential equation that $\omega$ describes on $Y'$ is in general a very difficult problem, though it has been successfully carried out in a few cases: on surfaces related to B\"uchi's problem, Vojta \cite{Vojta2000} determined an explicit symmetric differential, as well as a description of the solution curves to the corresponding differential equation to determine the genus $0$ and $1$ curves. This line of research was significantly expanded by Garc\'ia-Fritz~\cites{GarciaFritz,GarciaFritzIJNT} and by Garc\'ia-Fritz and Urz\'ua~\cite{GarciaFritzUrzua} for other surfaces, including the surface parametrizing perfect cuboids, for which they show that every curve of genus $0$ or $1$ must pass through at least $2$ nodes. The calculations in~\cite{BO} are asymptotic in $m$ and thus cannot be used to explicitly determine the locus in $Y$ containing the genus $0$ and $1$ curves.\\

\subsection{Contributions to the study of symmetric differentials}
This article contributes to the study of algebraic quasi-hyperbolicity in two ways.  First, we lay out \emph{explicit methods} for the calculation of the restrictions imposed by symmetric differentials on curves of genus $0$ or $1$ on nodal surfaces, and showcase our methods on specific surfaces (e.g., Barth's sextic and the surface parametrizing perfect cuboids).  Second, we give new, non-asymptotic lower bounds for the dimensions of spaces of symmetric differentials on resolutions of complete intersection surfaces with du Val (ADE) singularities.  These bounds allow us to increase the range of surfaces covered by~\cite{BO,RoulleauRousseau} that are known to be algebraically quasi-hyperbolic, which now includes, for example, Barth's Decic~\cite{Barth1996} and Sarti's surface~\cite{Sarti2001}. See Remark~\ref{rem:BORR} for a full discussion. They also allow us to show that the surface parametrizing $3\times 3$ magic squares of squares is algebraically quasi-hyperbolic.

\subsection{Set-up}
\label{ss:set-up}
For the rest of this section, we use the following notation: $X\subset \PP^n$ denotes a complex projective surface that is a complete intersection of multidegree $(d_1,\ldots,d_{n-2})$, with a singular locus $S$  consisting of $\ell$ isolated du Val singularities $\{s_1,\ldots,s_\ell\}$. We let $\tau\colon Y\to X$ be a minimal resolution of $X$, with exceptional locus $E$. We write $\hat{S}^m\Omega_X^1$ for the reflexive hull $(S^m\tau_*\Omega_Y^1)^{\vee\vee}$.

\subsection{Explicit methods}
Using graded modules over the coordinate ring of $X$, we explain how to  \emph{explicitly compute} a basis for the vector space $\HH^0(X,\hat{S}^m\Omega^1_X)$.  Current technology suffices to execute these ideas for small values of $m$. Using the identifications
\[
\HH^0(X,\hat{S}^m\Omega^1_X) \simeq \HH^0(X- S,S^m\Omega^1_X)\simeq \HH^0(Y- E,S^m\Omega^1_Y) 
\]
we study subspaces of sections of $\HH^0(X,\hat{S}^m\Omega^1_X)$ that can be extended to (at least part of) the exceptional divisor $E$ on $Y$.  See \S\ref{S:comp_reg_diff} for details.

As mentioned already, a section $\omega\in \HH^0(Y,S^m\Omega^1_Y)$ yields a degree $m$ first order differential equation that any genus $0$ curve $C\subset Y$ must satisfy, reflecting that $\omega$ must pull back to $0$ on $C$.  Two linearly independent sections $\omega_1$ and $\omega_2$ force $C$ to \emph{simultaneously} satisfy two differential equations. Whether this is possible can be expressed in terms of the vanishing of a resultant variety $\res(\omega_1,\omega_2)$; a precise definition of this variety is given in \S\ref{s:resultants}. If this resultant does not vanish identically, it furnishes a closed proper sublocus of $Y$ to which $C$ must belong.  An analysis of the irreducible components of $\res(\omega_1,\omega_2)$ then yields a complete list of genus $0$ curves on $Y$.

In concrete examples, it is possible that $\HH^0(Y,S^m\Omega^1_Y) = 0$ for small values of $m$.  Not all is lost.  Already two linearly independent sections $\omega_1$, $\omega_2\in \HH^0(X,\hat{S}^m\Omega^1_X)$ give rise to a closed sublocus $\res(\omega_1,\omega_2) \subset X$; if this closed set is strictly contained in $X$, then among its irreducible components one finds all \emph{complete} curves $C \subset X$ of genus $0$ that do not pass through any of the nodes of $X$.  

Crucially, intermediate subspaces of $\HH^0(Y,S^m\Omega^1_Y) \subset \HH^0(X,\hat{S}^m\Omega^1_X)$ can at once strengthen the above claims on genus $0$ curves on $X$ (or $Y$), as well as give constraints on the locus of genus $1$ curves on $X$.  For example, we show that if $X \subseteq \PP^3$ is a surface with $\ell$ isolated $A_1$ singularities, and if $H \subset \PP^3$ is a plane, then two linearly independent elements $\omega_1$ and $\omega_2$ in the intermediate subspace 
\[
\HH^0(Y,S^2\Omega^1_Y) \subset \HH^0\left(X - S,(S^2\Omega^1_X)(-H)\right) \subset \HH^0(X,\hat{S}^2\Omega^1_X)
\]
can be extended to regular differentials over the nodes $S\cap H$ (see Corollary~\ref{C:vanishing_on_hyperplane}), and thus their pullbacks $\tau^*\omega_1$ and $\tau^*\omega_2$ extend to regular differentials on the part of exceptional component $E$ lying over $S\cap H$. Let $S' = S - H$; we show that any complete genus $0$ curve $C\subset X - S'$ must lie in the closed proper subset $\res(\omega_1,\omega_2) \subset X$. In particular examples, we explicitly determine this locus and decompose it into irreducible components. Any genus $0$ curve on $X$ that passes only through nodes in $X\cap H$ must be among these components.

By varying the plane $H$ across the set of planes in $\PP^3$ spanned by any three $A_1$ singularities of $X$ we may conclude that any curve $C\subset X$ of genus $0$ must be among the curves we encountered or must pass through at least four non-coplanar nodes of $X$.  In addition, for any complete curve $C \subset X - S$ of genus $1$, the pull-back of the differentials $\omega_1$ and $\omega_2$ to $C$ each vanish on $C\cap H$, and thus they must both be identically zero on $C$. This forces the curve to be contained in the closed proper sublocus $\res(\omega_1,\omega_2)$.  If, on the other hand, the intersection $C\cap S$ is not empty, then $C$ must be contained in a linear space whose dimension is one more than that of the span of $C\cap S$ (see Corollary~\ref{C:explicit_resultant_locus}), and this forces $C$ to pass through at least two singularities in $C\cap S$.

As an example of this circle of ideas, we prove restrictions on the locus of genus $0$ or $1$ curves on Barth's sextic surface $X_6 \subset \PP^3_{\Q(\sqrt{5})}$, a surface whose singular locus consists of $65$ isolated $A_1$ singularities.

\begin{theorem}
\label{thm:BarthIntro}
Let $\phi=\frac{1}{2}(\sqrt{5}+1)$, and let $X_6 \subset \PP^3_{\Q(\sqrt{5})}$ be Barth's sextic surface, defined by
\[
X_6 : \quad 4(\phi^2x^2-y^2)(\phi^2y^2-z^2)(\phi^2z^2-x^2)-(1+2\phi)(x^2+y^2+z^2-w^2)=0.
\]
Any genus $0$ curve on $X_6$ must pass through at least four singularities.  Furthermore, there are exactly 27 genus $0$ curves on $X_6$ lying on planes spanned by singularities (they are listed in~\S\ref{ss:BarthSextic}); any genus $0$ curve on $X_6$ not among these $27$ curves must pass through at least four singularities that span $\PP^3$. Any genus $1$ must pass through at least two singularities and lie in a plane or pass through at least three non-collinear singularities.
\end{theorem}

We also obtain restrictions on the locus of genus $0$ and $1$ curves on the surface $\Xpc$ parametrizing \emph{perfect cuboids}, i.e., cuboids with all sides $x_1$, $x_2$, $x_3$, diagonals $y_1$, $y_2$, $y_3$ and body diagonal $z$ rational.  It is a complete intersection in $\PP^6$ of four quadrics: 
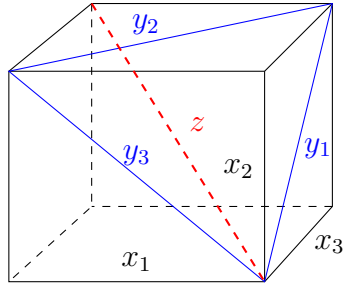
\begin{multicols}{2}
\begin{flushright}
\begin{tikzpicture}
\coordinate (A) at (1.7,1.4,1.3) ;
\coordinate (B) at (1.6,1.3,-1.3) ;
\coordinate (C) at (-1.7,1.4,1.3) ;
\coordinate (D) at (1.7,-1.4,1.3) ;
\coordinate (E) at (-1.6,1.3,-1.3) ;
\coordinate (F) at (-1.7,-1.4,1.3) ;
\coordinate (G) at (1.6,-1.4,-1.3) ;
\coordinate (H) at (-1.6,-1.4,-1.3) ;
\begin{scope}[every node/.style={midway}]
\draw (D) -- node[right=2pt] {$x_3$} (G) -- (B)  -- (E) -- (C)
   -- (F) -- node[above] {$x_1$} (D) ;
\draw (A) -- node[above=3pt, left] {$x_2$} (D) ;
\draw (B)  -- (A) -- (C) ;
\draw[dashed] (G) -- (H) ;
\draw[dashed] (E) -- (H) -- (F) ;
\draw[blue] (B) -- node[above=5pt, left] {$y_2$} (C)
   -- node[above] {$y_3$} (D) -- node[below=2pt, right=-2pt] {$y_1$} (B) ;
\draw[dashed, thick, red] (E) -- node[above right] {$z$} (D) ;
\end{scope}
\end{tikzpicture}
\end{flushright}

\begin{flushleft}
{
\[\Xpc : \quad
\left\{
\begin{split}
y_1^2&=x_2^2+x_3^2,\\
y_2^2&=x_3^2+x_1^2,\\
y_3^2&=x_1^2+x_2^2,\\
z^2&=x_1^2+x_2^2+x_3^2.
\end{split}\right.
\]
}

\end{flushleft}
\end{multicols}
\noindent The surface $\Xpc$ has 48 singularities of type $A_1$; it contains 32 plane conic curves and 60 genus $1$ curves identified by van Luijk in~\cite{vanLuijk2000} (see \S\ref{s:QuadricCIs}). We prove:

\begin{theorem}
\label{thm:CuboidIntro}
Let $\Xpc$ be the perfect cuboid surface. Any genus $0$ curve on $\Xpc$ must pass through at least six distinct singularities, Any genus $0$ curve on $\Xpc$ other than van Luijk's $32$ plane conics must pass through at least seven singularities that span $\PP^6$.

Any genus $1$ curve on $\Xpc$ lies in a linear space of at most one dimension higher than the linear space spanned by the singularities it passes through. In particular, any genus $1$ curve on $\Xpc$ passes through at least two singularities and is a component of a hyperplane section or passes through at least six singularities spanning a hyperplane.
\end{theorem}

\begin{remark} Garc\'ia-Fritz and Urz\'ua~\cite{GarciaFritzUrzua} study a natural composition of cyclic covers $\Xpc \to X_0 := \PP^1\times \PP^1$.  By investigating integral curves for the pullback of a section $\omega \in \HH^0(X_0,(S^2\Omega^1_{X_0})(2,2))$, they show that every curve of genus $0$ or $1$ on $\Xpc$ ust pass through at least $2$ nodes. Theorem~\ref{thm:CuboidIntro} gives stronger results vis-\`a-vis genus $0$ curves, but similar results for genus $1$ curves.
It would be interesting to see if their approach can be combined with our methods to produce even stronger results. 
\end{remark}

\subsection{Dimensions of spaces of symmetric differentials}
Keep the notation of~\S\ref{ss:set-up}.  We obtain further results on quasi-hyperbolicity of surfaces from another angle, by leveraging a pair of computable lower bounds for the dimensions $h^0(Y,S^m\Omega^1_Y)$ in terms of the Chern classes $c_1(Y)^2$ and $c_2(Y)$, as well as $\ell$, and $m$.  The general result is best phrased using the constituent terms $\chi^0$ and $\chi^1$ of Wahl's local Euler characteristic (see~\cite{Wahl1993} and \S\ref{S:local_euler_char}).  For a singularity $s \in X$, we denote by $X^\circ$ an analytic neighborhood of $s$, and let $Y^\circ$ be its inverse image under $\tau$; we write $E_s$ for the component of $E$ in $Y^\circ$ above $s$. The quantity $\chi^0(s,S^m\Omega^1_{Y})$ is the codimension of the subspace of $\HH^0(Y^\circ-E_s,S^m\Omega^1_{Y^\circ}) \isom \HH^0(X^\circ,\hat S^m\Omega^1_{X^\circ})$ of sections that extend over $s$; the quantity $\chi^1(s,S^m\Omega^1_{Y})$ is the dimension of $\HH^1(Y^\circ,S^m\Omega^1_{Y^\circ})$ around a node. 

\begin{theorem}
\label{thm:main}
Let $X$ be an irreducible complex projective surface whose singular locus $S$ is a finite set of isolated du Val singularities. Let $\tau\colon Y \to X$ be a minimal resolution.
Then for $m\geq 3$, we have
\begin{equation}\label{E:h0Yestimate_chi}
h^0(Y,S^m\Omega^1_Y)\geq \chi(Y,S^m\Omega^1_Y)+\sum_{s\in S}\chi^1(s,S^m\Omega^1_{Y}).
\end{equation}
Moreover, the inequality
\begin{equation}\label{E:h0Yestimate_h0}
h^0(Y,S^m\Omega^1_Y)\geq h^0(X,\hat S^m\Omega^1_X) -\sum_{s\in S}\chi^0(s,S^m\Omega^1_{Y})
\end{equation}
holds for $m \geq 1$.
\end{theorem}

Conceptually, the lower bound~\eqref{E:h0Yestimate_h0} records that the conditions imposed by individual singularities to extend a section of $\HH^0(X,\hat S^m\Omega^1_X)$ to one in $\HH^0(Y,S^m\Omega^1_Y)$ are at worst linearly independent.

The term $\chi(Y,S^m\Omega^1_Y)$ in the right hand side of~\eqref{E:h0Yestimate_chi} can be computed by combining a result of Atiyah~\cite{Atiyah1958} (see Lemma~\ref{L:atiyah}) with~\eqref{E:chiYSm}, which is a standard Chern class computation included in the Appendix. We get
\[
\chi(Y,S^m\Omega^1_Y)=\tfrac{1}{12}\big(
2(K^2-\chi)m^3-6\chi m^2-(K^2+3\chi)m+K^2+\chi
\big),
\]
where $K^2 = c_1(Y)$ and $\chi=c_2(Y)$.

In Propositions~\ref{P:codim_E_regular} and~\ref{P:local_chi}, we determine \emph{exact} expressions for $\chi^0(s,S^m\Omega^1_{Y})$ and $\chi^1(s,S^m\Omega^1_{Y})$ in the case where $s$ is an $A_1$ singularity:
\begin{align*}
\chi^0(s,S^m\Omega^1_{Y}) &= 
\begin{cases}
\frac{11}{108}m^3 + \frac{11}{36}m^2 + \frac{1}{6}m & m \equiv 0 \pmod 6 \\
\frac{11}{108}m^3 + \frac{11}{36}m^2 - \frac{1}{12}m - \frac{35}{108} & m \equiv 1 \pmod 6 \\
\frac{11}{108}m^3 + \frac{11}{36}m^2 + \frac{7}{18}m + \frac{5}{27} & m \equiv 2 \pmod 6 \\
\frac{11}{108}m^3 + \frac{11}{36}m^2 - \frac{1}{12}m - \frac{1}{4} & m \equiv 3 \pmod 6 \\
\frac{11}{108}m^3 + \frac{11}{36}m^2 + \frac{1}{6}m - \frac{2}{27} & m \equiv 4 \pmod 6 \\
\frac{11}{108}m^3 + \frac{11}{36}m^2 + \frac{5}{36}m - \frac{7}{108} & m \equiv 5 \pmod 6 
\end{cases},
\end{align*}
\begin{align*}
\chi^1(s,S^m\Omega^1_{Y}) &=
\begin{cases}
{\frac {4}{27}}\,{m}^{3}+\frac{4}{9}{m}^{2}+\frac{1}{3}m & \quad\  m\equiv 0\pmod{3}\\
{\frac {4}{27}}\,{m}^{3}+\frac{4}{9}{m}^{2}+\frac{1}{3}m+{\frac{2}{27}}& \quad\ m\equiv 1\pmod{3}\\
{\frac {4}{27}}\,{m}^{3}+\frac{4}{9}{m}^{2}+\frac{1}{9}m-{\frac{5}{27}}& \quad\ m\equiv 2\pmod{3}
\end{cases}.
\end{align*} 

Putting this all together, we obtain several algebraic quasi-hyperbolicity results, illustrated by the following examples.

\subsubsection{Nodal Surfaces in $\PP^3$}
If $X$ is a surface of degree $d$ in $\PP^3$, with 
\begin{equation}
\label{eq:P3lowerbound}
\ell > \frac{9}{4}(2d^2 - 5d)
\end{equation}
$A_1$ singularities, then the lower bound~\eqref{E:h0Yestimate_chi} shows that $h^0(Y,S^m\Omega^1_Y) > 0$ for all $m \gg 0$ and grows cubically with $m$, i.e., $Y$ has a big cotangent bundle.  On the other hand, Miyaoka's bounds~\cite[\S2.3]{MiyaokaBounds} on the number of quotient singularities on surfaces imply the inequality 
\[
\ell \leq \frac{4}{9}d(d-1)^2.
\]
Thus, $d = 10$ is the smallest degree $\geq 5$ for which there can exist a surface with sufficiently many nodes on which our results apply. Happily, such surfaces do exist: Barth's decic surface~\cite{Barth1996}, for which $(d,\ell) = (10,345)$, has a minimal resolution with a big cotangent bundle and is thus algebraically quasi-hyperbolic.  In this case, we can guarantee that $h^0(Y,S^m\Omega^1_Y) > 0$ once $m\geq 160$; in fact, we show that $h^0(Y,S^{160}\Omega^1_Y) \geq 15755$.  Sarti's surface~\cite{Sarti2001} satisfies $(d,\ell) = (12,600)$, and is also algebraically quasi-hyperbolic; in this case, we can guarantee that  $h^0(Y,S^m\Omega^1_Y) > 0$ once $m\geq 28$, and that $h^0(Y,S^{28}\Omega^1_Y)\geq 7646$.

We note that the results in~\cite{RoulleauRousseau}, as well a corrected version of those in~\cite{BO} do not suffice to prove algebraic quasi-hyperbolicity of Barth's Decic or Sarti's surface. Indeed, the orbifold methods in~\cite{RoulleauRousseau} yield the slightly weaker lower bound $\ell > \frac{8}{3} (2d^2-5d)$ in place of~\eqref{eq:P3lowerbound}, while a corrected version of the calculations in~\cite{BO} yield the lower bound $\ell > \frac{36}{11}(2d^2-5d)$ in place of~\eqref{eq:P3lowerbound}. See Remark~\ref{rem:BORR} for a more thorough comparison of results.

Segre~\cite{Segre1947} constructed hypersurfaces of even degree $d$ in $\PP^3$ with $\ell=\frac{1}{4}d^2(d-1)$ nodal singularities by taking an equation of the form
\[G^2+\lambda\prod_{i=1}^dL_i,\]
where $G$ is a form of degree $d/2$ and $L_i$ are linear forms, and $\lambda$ is a scalar (see also~\cite{Beauville1980}*{p.~208}). For $d\geq 18$, this satisfies the bound~\eqref{eq:P3lowerbound}.

\subsubsection{Nodal Complete Intersections of Quadrics}
\label{ex:CIsIntro}

If $X$ is a complete intersection of $n-2$ quadrics in $\PP^n$ with $\ell$ isolated $A_1$ singularities, then we use the lower bound~\eqref{E:h0Yestimate_chi} to show that the resolution $Y$ has big cotangent bundle and is algebraically quasi-hyperbolic for $\ell \geq \ell_{\mathrm{min}}(n)$, where $\ell_{\mathrm{min}}(n)$ is defined by
\[
\begin{array}{c||c|c|c|c|c}
n&6&7&8&9&\geq 10\\
\hline
\ell_\mathrm{min}(n)&73&145&217&145&0\\
\end{array}\]
The fact that a $2$-dimensional complete intersection of quadrics with isolated du Val singularities in a sufficiently high-dimensional projective space has big cotangent bundle follows already from work of Roulleau and Rousseau \cite{RoulleauRousseau}, as such surfaces have positive second Segre class~\cite{Miyaoka}.

As an application, we deduce that a certain surface related to magic squares is algebraically quasi-hyperbolic.  Recall that an $n\times n$ \defi{magic square} is an $n\times n$ grid, filled with distinct positive integers, whose rows, columns, and diagonals add up to the same number.  It is unknown if there exists a $3\times 3$ magic square whose entries are distinct nonzero squares.  

\[
\def\arraystretch{1.3}
\begin{array}{|c|c|c|}
\hline
x_1^2&x_2^2&x_3^2\\
\hline
x_4^2&x_5^2&x_6^2\\
\hline
x_7^2&x_8^2&x_9^2\\
\hline
\end{array}
\]

\smallskip
\noindent Such a square gives rise to a rational point with nonzero coordinates on the complete intersection surface $\Xms\subset \PP^8$ defined by the relations
\begin{align*}
x_1^2 + x_2^2 + x_3^2 &= x_4^2 + x_5^2 + x_6^2 = x_7^2 + x_8^2 + x_9^2 = x_1^2 + x_4^2 + x_7^2 \\
&= x_2^2 + x_5^2 + x_8^2 = x_1^2 + x_5^2 + x_9^2 = x_3^2 + x_5^2 + x_7^2.
\end{align*}
This surface is smooth except for $256$ isolated ordinary double points.  This exceeds $\ell_\text{min}(8)=217$, so we obtain the following result.

\begin{thm}
\label{thm:magicsquares}
The complex projective surface $\Xms\subset \PP^8$ that parametrizes $3\times 3$ magic squares of squares is algebraically quasi-hyperbolic. 
\qed
\end{thm}
In fact, using \eqref{E:h0Yestimate_chi} we find that for $m\geq 47$ there are global sections and that $H^0(\Yms,S^{47}\Omega^1_{\Yms})\geq8448$.

\subsubsection{Partial Information}
Even in cases where Theorem~\ref{thm:main} cannot quite prove quasi-hyperbolicity of a surface, we can use the ideas behind its proof to determine restrictions on the properties of genus $0$ and $1$ curves on $Y$; see Proposition~\ref{C:regdif_lowerbound_quasiproj} and Proposition~\ref{P:foliation_quasihyperbolicity}.  For instance, if the set $S$ consists of $\ell$ isolated $A_1$ singularities, and if for some $0\leq r<\ell$ there is a constant $C>0$ such that
\[\chi(Y,S^m\Omega^1_Y)+\ell\chi^1(s,S^m\Omega^1_Y)+r\chi^0(s,S^m\Omega^1_Y)\sim Cm^3,
\]
then $X$ contains only finitely many genus $0$ or $1$ curves that pass through less than $r$ singularities of type $A_1$.

\subsection*{Acknowledgements} We thank Fedor Bogomolov for 
suggesting the idea to combine modern computational methods with the ideas of~\cite{BO} to study the locus of genus $0$ and $1$ curves in algebraically quasi-hyperbolic surfaces.  We also thank him for bringing \cite{Thomas} to the attention of the first and third author.  We thank Fr\'ed\'eric Campana for a useful conversation and for introducing the first and third authors to the work of Roulleau and Rousseau~\cite{RoulleauRousseau}. Finally, we would like to thank the anonymous referees for a careful reading of the manuscript and for making constructive suggestions to improve the exposition of the article. The first author acknowledges the support of the Natural Sciences and Engineering Research Council of Canada (NSERC), funding reference number RGPIN-2018-04191. The third author was partially supported by NSF grants DMS-1352291 and DMS-1902274.

\section{Global Symmetric Differentials I}
\label{S:GlobalSymmetricDifferentials-I}

Throughout this section we keep the notation of \S\ref{ss:set-up}. In particular, $\tau\colon Y \to X$ is a minimal resolution of a complex complete intersection surface $X \subset \PP^n$ of multidegree $(d_1,\ldots,d_{n-2})$ with at worst a finite set $S$ of isolated du Val singularities.

\begin{lemma}\label{L:atiyah}
Let $Z \subset \PP^n$ be a complex nonsingular complete intersection of multidegree $(d_1,\ldots,d_{n-2})$. Then
\[\chi(Y,S^m\Omega^1_Y)=\chi(Z,S^m\Omega^1_Z).\]
\end{lemma}
\begin{proof}
	This is a direct consequence of a beautiful result \cite{Atiyah1958} in the diffeomorphic category. We can take $X$ as the central member of a family $X_t$ of complete intersections, with the general member nonsingular. Then \cite{Atiyah1958}*{Theorem~3} gives that minimal resolutions of fibres are pairwise diffeomorphic. The result now follows from comparing the central fibre $X$ with a general member $Z$, because Euler characteristics are invariant under diffeomorphisms.
\end{proof}

Henceforth, in this section we assume that $d_1+\cdots+d_{n-2}>n+1$.

\begin{lemma} \label{L:Y_general_type} The surface $Y$ is of general type, i.e., its canonical class $K_Y$ is big.
\end{lemma}

\begin{proof}
Since the singularities of $X$ are du Val, it follows that $K_Y = \tau^*K_X$; see, e.g., the proof of~\cite[Theorem~4.20]{KollarMori}. The hypothesis $d_1+\cdots+d_{n-2}>n+1$ ensures that $K_X$ is big, and hence so is $K_Y$ by~\cite[Lemma~2.2.43]{Lazarsfeld}.
\end{proof}

In what follows, we begin our systematic study of lower bounds for the space of global sections of $\calA := S^m\Omega_Y^1$, in terms of $m$. Consider $\tau_*\calA$ and its reflexive hull $\hat{\calA}=(\tau_*\calA)^{\vee\vee}$. The Leray spectral sequence
\[
E_2^{p,q} := \HH^p(X,R^q\tau_*\calA) \implies \HH^{p+q}(Y,\calA)
\]
gives rise to the $6$-term exact sequence of low-degree terms
\[
\begin{split}
0\to \HH^1(X,\tau_*\calA) &\to \HH^1(Y,\calA)\to \HH^0(X,R^1\tau_*\calA)\to \HH^2(X,\tau_*\calA)\\
&\to \ker\left(\HH^2(Y,\calA)\to \HH^0(X,R^2\tau_*\calA)\right) \to \HH^1(X,R^1\tau_*\calA).
\end{split}
\]
The sheaf $R^1\tau_*\calA$ is supported on the $0$-dimensional scheme $S$, since $\tau$ is an isomorphism outside of $S$. Hence we have 
\[
\HH^1(X,R^1\tau_*\calA) = \HH^2(X,R^1\tau_*\calA) = 0.
\]
Inspecting page $2$ of the spectral sequence, this last equality shows that $\HH^0(X,R^2\tau_*\calA) = 0$ as well. Furthermore, since $\tau$ is an isomorphism outside $S$ and $\calA$ is reflexive, we see that $\hat{\calA}/\tau_*\calA$ and the kernel of $\tau_*\calA\to\hat\calA$ are both supported on $S$, which is $0$-dimensional, so $\HH^2(X,\tau_*\calA)=\HH^2(X,\hat\calA)$. We simplify our sequence to
\begin{equation}
\label{eq:LeraySimplified}
0\to \HH^1(X,\tau_*\calA) \to \HH^1(Y,\calA)\to \HH^0(X,R^1\tau_*\calA)\to \HH^2(X,\hat{\calA}) \to \HH^2(Y,\calA)\to 0.
\end{equation}

\begin{lemma}
\label{lem:Leray}
With notation as above, for $m\geq 3$ we have $h^2(Y,S^m\Omega_Y^1) = 0$ and
\[
h^1(Y,S^m\Omega_Y^1) = h^1(X,\tau_*S^m\Omega^1_Y)+h^0(X,R^1\tau_*S^m\Omega^1_Y).
\]
\end{lemma}

\begin{proof}
By~\cite[Proposition~2.3]{BO} (or \cite[Lemme~3.3.2]{Deschamps}) and Lemma~\ref{L:Y_general_type}, we have that $h^2(X,\hat\calA) = 0$ for $m\geq 3$. The lemma now follows by looking at dimensions on~\eqref{eq:LeraySimplified}.
\end{proof}

\begin{cor}
\label{cor:MainLowerBound}
The inequality 
\[
h^0(Y,S^m\Omega_Y^1) \geq \chi(Y,S^m\Omega^1_Y) + h^0(X,R^1\tau_*S^m\Omega^1_Y),
\]
holds for $m\geq 3$. 
\qed
\end{cor}

\begin{proof}
Since $h^2(Y,S^m\Omega_Y^1) = 0$ for $m\geq 3$, we have
\[
h^0(Y,S^m\Omega_Y^1) = \chi(Y,S^m\Omega_Y^1) + h^1(Y,S^m\Omega_Y^1)
\]
The conclusion now follows from Lemma~\ref{lem:Leray} and the crude estimate $h^1(X,\tau_*S^m\Omega^1_Y) \geq 0$.
\end{proof}

\begin{remark}
We expect that improving the coarse estimate $h^1(X,\tau_*S^m\Omega^1_Y) \geq 0$ in the proof of Corollary~\ref{cor:MainLowerBound} would significantly strengthen our results. 
\end{remark}

In the following sections, we compute the right hand side of the inequality in Corollary~\ref{cor:MainLowerBound} exactly in the case where $S$ consists only of $A_1$ singularities. The Euler characteristic $\chi(Y,S^m\Omega_Y^1)$ is easily computed using Lemma~\ref{L:atiyah} and \eqref{E:chiYSm}, taking into account that for a nonsingular multidegree $(d_1,\ldots,d_{n-2})$ complete intersection $Z$, we have
\[K_Z^2=(n+1-\sigma_1)^2d,\; \chi_Z=\left(\tbinom{n+1}{2}-(n+1-\sigma_1)\sigma_1-\sigma_2\right)d,\]
where $d=\prod_id_i$, $\sigma_1=\sum_i d_i$, and $\sigma_2=\sum_{i<j} d_id_j$. Since $R^1\tau_*S^m\Omega^1_Y$ is supported on the $0$-dimensional scheme $S$, we compute $h^0(X,R^1\tau_*S^m\Omega^1_Y)$ point by point, restricting to sufficiently small neighborhoods around them. This requires a detailed study of local Euler characteristics, which we address in Section~\ref{S:local_euler_char}.

\section{Local Euler Characteristics}
\label{S:local_euler_char}
Let $(X^\circ,s)$ be an isolated normal analytic complex surface singularity, and let $(Y^\circ,E_s)$ be a \defi{good resolution} of $X^\circ$, by which we mean a resolution with a simple normal crossings divisor $E$.  For a locally free coherent sheaf $\calFA$ on $Y^0$, following Wahl~\cite{Wahl1993}, define the \defi{local Euler characteristic} at $s \in X^0$ by
\begin{equation}
\label{eq:localEuler}
\begin{split}
\chi^0(s,\calFA) &:= \dim\left[\HH^0(Y^\circ - E_s,\calFA)/\HH^0(Y^\circ,\calFA)\right],\\
\chi^1(s,\calFA) &:= h^1(Y^\circ,\calFA),\\
\chi(s,\calFA) &:= \chi^0(s,\calFA)+\chi^1(s,\calFA).
\end{split}
\end{equation}

\subsection{Proof of Theorem~\ref{thm:main}}
We now have all the necessary ingredients and notation to prove Theorem~\ref{thm:main}.  Recall that in the statement of the theorem, the morphism $\tau\colon Y \to X$ follows the conventions of \S\ref{ss:set-up}.
By Corollary~\ref{cor:MainLowerBound}, we know that 
\[h^0(Y,S^m\Omega^1_Y)\geq
\chi(Y,S^m\Omega^1_Y)+h^0(X,R^1\tau_*S^m(\Omega^1_Y)).\]
The sheaf $R^1\tau_*S^m(\Omega^1_Y)$ is supported on a $0$-dimensional scheme $S$, so $h^0(X,R^1\tau_*S^m(\Omega^1_Y))$ is simply the sum of the contributions at each $s\in S$. We get that this contribution is  $\chi^1(s,S^m\Omega^1_Y)$ by restricting to a sufficiently small affine neighbourhood $X^\circ$ of $s$. This proves~\eqref{E:h0Yestimate_chi}.

Since $\tau\colon (Y-E)\to (X-S)$ is an isomorphism and $S$ is of codimension $2$, we see that
\[\HH^0(X,\hat S^m(\Omega^1_X))\simeq \HH^0(X-S,S^m\Omega^1_{X-S})\simeq \HH^0(Y-E,S^m\Omega^1_Y).\]
By definition, $\chi^0(s,S^m\Omega^1_Y)$ measures exactly the codimension for each singularity separately. At the worst, each of these singularities imposes independent linear conditions on sections in $\HH^0(Y-E,S^m\Omega^1_Y)$ to extend into each component of $E$, giving \eqref{E:h0Yestimate_h0}.
\qed \\

Requiring regularity on only \emph{some} components of the exceptional divisor yields stronger lower bounds.  We illustrate this in the case that $S$ consists entirely of $A_1$ singularities.

\begin{prop}\label{C:regdif_lowerbound_quasiproj}
With notation as in Theorem~\ref{thm:main}, assume further that $S$ consists of $\ell$ isolated $A_1$ singularities.
Let $E_1,\ldots,E_r$ be exceptional components on $Y$ above $r\leq \ell$ of the elements of $S$. Then for $m\geq 3$ we have
\begin{equation}\label{E:h0Yestimate_chi_r}
h^0(Y-(E_1\cup\cdots\cup E_r),S^m\Omega^1_Y)\geq \chi(Y,S^m\Omega^1_Y)+\ell\chi^1(s,S^m\Omega^1_Y)+r\chi^0(s,S^m\Omega^1_Y),
\end{equation}
and in fact for all $m\geq 1$ that
\begin{equation}\label{E:h0Yestimate_h0_r}
h^0(Y-(E_1\cup\cdots\cup E_r),S^m\Omega^1Y)\geq h^0(X,(S^m(\Omega^1_X))^{\vee\vee}) -(\ell-r)\chi^0(s,S^m\Omega^1_Y).
\end{equation}
\end{prop}

\subsection{Singularities of type $A_1$}
In this section, we compute local Euler characteristics for sheaves associated to symmetric differentials in the case where $s$ is an $A_1$ singularity.  A model for $X^\circ$ is a quadric cone $x_1x_3=x_2^2$ in $\AA^3$ with $s=(0,0,0)$, and $Y^\circ$ is the blow-up at the vertex, so that $E_s\simeq \PP^1$. The assignment $(t,u)\mapsto (t,tu,tu^2)$ is an affine chart of $\tau\colon Y^\circ\to X^\circ$, where the exceptional fiber $E_s$ is given by $t=0$.

We consider the sheaves $\calA=S^m\Omega^1_{Y^\circ}$ and $\calB_h=(S^m(\Omega^1_{Y^\circ}(\log E_s)))(-hE_s)$. They agree on $Y^\circ-E_s$, so $\tau_*\calA|_{X^\circ-s}=\tau_*\calB_h|_{X^\circ-s}$. Since $s$ has codimension $2$ in $X^\circ$, it follows that the reflexive hulls agree, i.e.,
\[\hat\calA:=(\tau_*\calA)^{\vee\vee}=(\tau_*\calB_h)^{\vee\vee}.\]
Ultimately, we will compute $\chi(s,S^m\Omega^1_{Y^\circ})$ by understanding for which values of $h$ we have $\chi(s,\calB_h) = 0$.

The singularity $(X^\circ,s)$ can also be viewed as a quotient singularity arising from the degree $2$ finite cover $f\colon X'\to X^\circ$, where $X'=\AA^2$ and $f$ is given by $(z_1,z_2)\mapsto(z_1^2,z_1z_2,z_2^2)$ with automorphism $\iota\colon (z_1,z_2)\mapsto(-z_1,-z_2)$. The unique fixed point and preimage of $s$ is $s'=(0,0)$.

Since $X'\to X^\circ$ is a finite quotient map with automorphism group $\langle \iota\rangle$, sections in $\HH^0(X^\circ-s,S^m\Omega^1_{X^\circ})$ pull back to sections in $\HH^0(X'-s',S^m\Omega^1_{X'})^{\iota}$, which by purity extend into the nonsingular point $s'$. Hence the vector space $\HH^0(Y^\circ-E_s,S^m\Omega^1_{Y^\circ}) \simeq \HH^0(X^\circ-s,S^m\Omega^1_{X^\circ})$ is naturally isomorphic to $\HH^0(X',S^m\Omega^1_{X'})^{\iota}$; we identify these spaces from now on.

The ring
\[\bigoplus_{m\geq 0} \HH^0(X',S^m\Omega^1_{X'})\]
is isomorphic to the polynomial ring $k[z_1,z_2,dz_1,dz_2]$, bigraded by the total degrees in $z_1,z_2$ and $dz_1, dz_2$ respectively, with graded parts
\[
V_{m,n}=\langle z_1^jz_2^{n-j}dz_1^idz_2^{m-i}: i=0,\ldots,m; j=0,\ldots,n \rangle.
\]
For the $\iota$-invariant subring we have
$\bigoplus_m\HH^0(X',S^m\Omega^1_{X'})^{\iota}=\bigoplus_{n\equiv m\pmod 2} V_{m,n}$.
The identification $\HH^0(Y^\circ-E_s,S^m\Omega^1_{Y^\circ})\simeq \HH^0(X',S^m\Omega^1_{X'})^{\iota}$ induces a valuation $\ord_E$ on the latter, which extends to all of $\HH^0(X',S^m\Omega^1_{X'})$ as a valuation taking values in $\frac{1}{2}\Z$. We describe the valuation on $k[z_1,z_2,dz_1,dz_2]$ by introducing a square root of $t$, denoted by $t^{1/2}$. The relations $z_1^2=t, z_1z_2=tu, z_2^2=tu^2$ give rise to relations between their derivatives as well, which can be expressed as a ring homomorphism
\[k[z_1,z_2,dz_1,dz_2]\to k(t^{1/2})[u,dt,du]\]
defined by
\[z_1\mapsto t^{1/2},\;z_2\mapsto t^{1/2}u,\;
dz_1\mapsto \tfrac{1}{2}t^{-1/2}dt,\; dz_2\mapsto \tfrac{1}{2}(t^{-1/2}udt+2t^{1/2}du).
\]
The valuation is the obvious one with respect to $t$ on $k(t^{1/2})[u,dt,du]$, pulled back along this homomorphism. In particular, we have
\[
\ord_E(z_1) = \ord_E(z_2) = \frac{1}{2},\;\;\text{and}\;\;
\ord_E(dz_1) = \ord_E(dz_2) = -\frac{1}{2}.
\]

\begin{lemma}\label{L:chi0B} We have $\chi^0(s,\calB_h)=0$ if and only if $h <(m+1)/2$.
\end{lemma}
\begin{proof}
We observe that $\HH^0(Y^\circ-E_s,\calB_h)\subset\HH^0(Y^\circ-E_s,S^m\Omega_{Y^\circ}^1)\subset k[z_1,z_2,dz_1,dz_2]$, where the last inclusion comes from the identification explained above. On the affine patch $Y'$ of $Y^\circ$ with coordinates $(t,u)$, we see that $\HH^0(Y',S^m(\Omega^1_{Y'}(\log E_s)))$ is a free $k[t,u]$-module with basis $(dt/t)^m,(dt/t)^{m-1}du,\ldots,(dt/t)(du)^{m-1},(du)^m$.
If we also consider the complementary patch $(s,v)=(tu^2,1/u)$ of $Y^\circ$, we see that 
\[
(dt/t)=d(sv^2)/(sv^2)=(v^2ds+2vsdv)/(sv^2)=ds/s+2dv/v,
\] which is not a log-differential: we would need to multiply it by $t$. Note that
\[
t(dt/t)^2=v^2s(ds/s)^2+4vsdv(ds/s)+4s(dv)^2
\]
is a log-differential on all of $Y^\circ$. Inside $\HH^0(Y',S^m(\Omega^1_{Y'}(\log E_s)))$ we can characterize the elements of $\HH^0(Y^\circ,S^m(\Omega^1_{Y^\circ}(\log E_s)))$ as those forms for which the coefficient of $(dt/t)^i(ds)^{m-i}$ is divisible by $t^{\ceil{ i/2}}$. This coincides with $\HH^0(X',S^m\Omega^1_{X'})^\iota$. We see that for $h\leq m-\ceil{m/2} < (m+1)/2$ we have $\HH^0(Y^\circ,\calB_h)=\HH^0(Y^\circ-E_s,\calB_h)$. Furthermore $(dz_1)^m$ for even $m$ and $z_1(dz_1)^m$ for odd $m$, show that for larger $h$, equality does not hold.
\end{proof}

We use the valuation $\ord_E$ to determine $\chi^0(s,\calA) = \chi^0(s,S^m\Omega^1_{Y^\circ})$ as a function of $m$ in the following proposition. From the leading coefficient one can read off the corrected asymptotics for \cite{BO}*{Lemma~2.2} as well.
\begin{proposition}\label{P:codim_E_regular} 
For an $A_1$ singularity $(X^\circ,s)$ and a minimal resolution $\tau\colon Y^\circ\to X^\circ$ we have
	\begin{equation}
	\label{eq:localTerms}
	\chi^0(s,S^m\Omega^1_{Y^\circ}) = 
	\begin{cases}
	\frac{11}{108}m^3 + \frac{11}{36}m^2 + \frac{1}{6}m & m \equiv 0 \pmod 6 \\
	\frac{11}{108}m^3 + \frac{11}{36}m^2 - \frac{1}{12}m - \frac{35}{108} & m \equiv 1 \pmod 6 \\
	\frac{11}{108}m^3 + \frac{11}{36}m^2 + \frac{7}{18}m + \frac{5}{27} & m \equiv 2 \pmod 6 \\
	\frac{11}{108}m^3 + \frac{11}{36}m^2 - \frac{1}{12}m - \frac{1}{4} & m \equiv 3 \pmod 6 \\
	\frac{11}{108}m^3 + \frac{11}{36}m^2 + \frac{1}{6}m - \frac{2}{27} & m \equiv 4 \pmod 6 \\
	\frac{11}{108}m^3 + \frac{11}{36}m^2 + \frac{5}{36}m - \frac{7}{108} & m \equiv 5 \pmod 6 
	\end{cases}
	\end{equation}
	In particular, the first few values we get are
	\[
	\begin{array}{c|cccccccccccc}
	m&1&2&3&4&5&6&7&8&9&10&11&12\\
	\hline
	\mathrm{codim}&0&3&5& 12& 21& 34& 49& 75& 98& 134& 174& 222 
	\end{array}
	\]
\end{proposition}

\begin{proof}
	We write $W_{m,n}:=V_{m,n}\cap \HH^0(Y^\circ,\calA)$.
	It follows immediately that $W_{m,n}=V_{m,n}$ if $n\geq m$.
	In addition, we see that $\ord_E(z_1dz_2-z_2dz_1)=\ord_E(tdu)=1$. In fact, by looking at leading terms with respect to $u$ and $dt$, we see that $\omega\in V_{m,n}$ has $\ord_E(\omega)>(n-m)/2$ if and only if $\omega$ is divisible by $z_1dz_2-z_2dz_1$.
	By applying this criterion iteratively we find that
	\[W_{m,n}=\begin{cases}
	V_{m,n}&\text{ if }n\geq m\\
	V_{(m+n)/2,(3n-m)/2}(z_1dz_2-z_2dz_1)^{(m-n)/2}&\text{ if } m/3\leq n<m\\
	0&\text{ if }n<m/3
	\end{cases}\]
	Since $z_1dz_2-z_2dz_1$ is bihomogeneous, it follows that
	\[
	\HH^0(Y^\circ - E_s,\calA)/\HH^0(Y^\circ,\calA)=\bigoplus_{n\equiv m\pmod{2}} V_{m,n}/W_{m,n}.
	\]
	Using that $\dim V_{m,n}=(m+1)(n+1)$ and hence $\dim W_{m,n}=(n+m+2)(3n-m+2)/4$ for $m/3\leq n\leq m$ we can find the formulas by straightforward summation.
\end{proof}

\begin{cor}\label{C:vanishing_on_hyperplane} If $\omega\in \HH^0(X^\circ-s,S^m\Omega^1_{X^\circ}(-\floor{\frac{1}{2}m} H))$, where $H$ is a hyperplane section containing $s$, then $\tau^*\omega$ extends to a regular differential on the component $E_s$ of the exceptional divisor of $Y$ lying over $s$.
\end{cor}
\begin{proof}
The form $\omega$ pulls back to $\bigoplus_{n\geq m} V_{m,n}$ and therefore lies in $\bigoplus W_{m,n}$.
\end{proof}

\begin{lemma}\label{L:chi1B} We have $\chi^1(s,\calB_h)=0$ if and only if $h>(m-2)/2$.
\end{lemma}
\begin{proof}
The $4$-term exact sequence associated to the Leray spectral sequence for $\tau\colon Y^\circ \to X^\circ$ and the sheaf $\calB_h$ is
\[
0 \to \HH^1(X^\circ,\tau_*\calB_h) \to \HH^1(Y^\circ,\calB_h) \to \HH^0(X^\circ,R^1\tau_*\calB_h) \to \HH^2(X^\circ,\tau_*\calB_h).
\]
The morphism $\tau$ being proper to a locally Noetherian base, the sheaves $R^i\tau_*\calB_h$
 are coherent~\cite[III.1~Th\'eor\`eme 3.2.1]{EGAIII}. Since without loss of generality we can take $X^\circ$ to be affine, we have
 \[
\HH^1(X^\circ,\tau_*\calB_h) = \HH^2(X^\circ,\tau_*\calB_h) = 0;
 \]
 see~\cite[III.3.5]{Hartshorne}. This shows that
 \[
\HH^1(Y^\circ,\calB_h) \xrightarrow{\sim} \HH^0(X^\circ,R^1\tau_*\calB_h).
 \]
To complete the proof, we show that $R^1\tau_*\calB_h = 0$ precisely when $h > (m-2)/2$. The sheaf $R^1\tau_*\calB_h$ is supported on $s$, so it is enough to understand its stalk $(R^1\tau_*\calB_h)_s$. By the Theorem on Formal Functions~\cite[III.11.1]{Hartshorne}, we have
\begin{equation}
\label{eq:formalfunctions}
(R^1\tau_*\calB_h)^{\wedge}_s \xrightarrow{\sim} \varprojlim_n \HH^1(nE,\calB_h),
\end{equation}
where $nE = Y^\circ\times_{X^\circ} \Spec(\calO_s/\frakm_s^n)$ and by abuse of notation the sheaf $\calB_h$ on the right hand side is the pullback of $\calB_h$ via the projection $nE \to Y^\circ$. Tensoring the exact sequence of sheaves \[
0 \to \calO_E(-nE) \to \calO_{nE} \to \calO_{(n+1)E} \to 0
\]
with the locally free sheaf $\calB_h$ and taking cohomology we obtain the exact sequence
\[
\HH^1(E,\calO_E(-nE)\otimes\calB_h) \to \HH^1(nE,\calO_{nE}\otimes \calB_h) \to \HH^1((n+1)E,\calO_{(n+1)E}\otimes\calB_h) \to 0.
\]
If $\HH^1(E,\calO_E(-nE)\otimes\calB_h) = 0$ for all $n \geq 0$, then
\[
\HH^1(nE,\calO_{nE}\otimes \calB_h) \xrightarrow{\sim} \HH^1((n+1)E,\calO_{(n+1)E}\otimes\calB_h)
\]
for all $n\geq 0$, which implies in turn that the projective limit in~\eqref{eq:formalfunctions} is isomorphic to $\HH^1(nE,\calO_{nE}\otimes \calB_h)$ for all $n$, and in particular, it is isomorphic to $\HH^1(E,\calO_{E}\otimes \calB_h)$.

To understand the cohomology groups $\HH^1(E,\calO_E(-nE)\otimes\calB_h)$, we use the residue exact sequence
\[
0 \to \Omega_E  \to \Omega_{Y^\circ}(\log E)|_E \to \calO_E \to 0.
\]
This sequence does not split~\cite[3.3]{Wahl1976}, so $\Omega_{Y^\circ}(\log E)|_E$ is isomorphic to the nontrivial class in $\Ext^1_{\calO_{\PP^1}}(\calO(-2),\calO)$, i.e.,
\begin{equation}
\label{eq:nonsplit}
\Omega_{Y^\circ}(\log E)|_E \simeq \calO_E(-1)\oplus \calO_E(-1).
\end{equation}
Taking into account that $\calO_Y(-hE)|_E \simeq \calO_E(2h)$ because $E^2 = -2$, equation \eqref{eq:nonsplit} shows that the restriction of $\calB_h$ to $E \simeq \PP^1$ is $\calO_{\PP^1}(-m+2h)^{\oplus(m+1)}$. We have
\begin{align*}
\HH^1(E,\calO_E(-nE)\otimes\calB_h) &\simeq \HH^1(\PP^1,\calO_{\PP^1}({2n}-m+2h)^{\oplus(m+1)}) \\
&\simeq \HH^0(\PP^1,\calO_{\PP^1}({-2n}+m-2h-2)^{\oplus(m+1)}),
\end{align*}
where the last isomorphism follows from Serre duality. This cohomology group vanishes for $n\geq 0$ precisely when $h > (m-2)/2$, in which case the $n=0$ vanishing shows that $\HH^1(E,\calO_{E}\otimes \calB_h) = 0$, and thus the projective limit~\eqref{eq:formalfunctions} vanishes as well.
\end{proof}

Lemmas~\ref{L:chi0B} and \ref{L:chi1B} combine to the following result.
\begin{cor}\label{C:chiB} We have $\chi(s,\calB_h)=0$ precisely for $h=\ceil{\frac{m}{2}}$.
\hfill\qed
\end{cor}

The vanishing result in Corollary~\ref{C:chiB} allows us to compute $\chi(s,S^m\Omega^1_{Y^\circ})$ and $\chi^1(s,S^m\Omega^1_{Y^\circ})$ for an $A_1$ singularity.

\begin{proposition}\label{P:local_chi}
For an $A_1$ singularity $(X^\circ,s)$ and a minimal resolution $\tau\colon Y^\circ\to X^\circ$ we have
	\begin{equation}
		\label{eq:locsymEuler}
		\chi(s,S^m\Omega_{Y^\circ}^1) = 
		\begin{cases}
			\displaystyle
			\tfrac{1}{4}m(m+1)(m+2) & \text{ if }m \equiv 0\pmod{2}, \\
			\displaystyle
			\tfrac{1}{4}(m+1)(m^2 + 2m - 1) & \text{ if }m \equiv 1\pmod{2}
		\end{cases}
	\end{equation}
and
\begin{equation}\label{E:local_chi1}
\chi^1(s,S^m\Omega^1_{Y^\circ})=
\begin{cases}
{\frac {4}{27}}\,{m}^{3}+\frac{4}{9}{m}^{2}+\frac{1}{3}m&\text{ if }m\equiv 0\pmod{3}\\
{\frac {4}{27}}\,{m}^{3}+\frac{4}{9}{m}^{2}+\frac{1}{3}m+{\frac{2}{27}}&\text{ if }m\equiv 1\pmod{3}\\
{\frac {4}{27}}\,{m}^{3}+\frac{4}{9}{m}^{2}+\frac{1}{9}m-{\frac{5}{27}}&\text{ if }m\equiv 2\pmod{3}
\end{cases}
\end{equation}
\end{proposition}

\begin{proof}
We choose a completion $X$ of $X^\circ$ such that the singular locus of $X$ consists of just $s\in X^\circ\subset X$. We take $\tau\colon Y\to X$ a minimal resolution. Then $Y^\circ$ is isomorphic to the inverse image of $X^\circ$ in $Y$ and $S^{m}(\Omega^1_{Y^\circ})$ is the restriction of $S^{m}(\Omega^1_{Y})$ to $Y^\circ$.

We consider the sheaves $\calA=S^m\Omega^1_{Y}$ and $\calB=(S^m(\Omega^1_{Y}(\log E_s)))(-hE_s)$, with $h=\ceil{ \frac{m}{2}}$. 

The sheaf $\calB$ is locally free and therefore reflexive, so
by~\cite[Lemma on p.\ 30]{Blache}, for $\hat\calB := (\tau_*\calB)^{\vee\vee}$ we know that
\[
\chi(X,\hat\calB) = \chi(Y,\calB) + \chi(s,\calB).
\]
Corollary~\ref{C:chiB} says that $\chi(s,\calB) = 0$, so $\chi(X,\hat\calB) = \chi(Y,\calB)$. Since $\hat\calA:=(\tau_*\calA)^{\vee\vee}=(\tau_*\calB)^{\vee\vee} = \hat\calB$, we obtain
\begin{align*}
\chi(s,\calA) &= \chi(X,\hat\calA) - \chi(Y,\calA) \\
&= \chi(X,\hat\calB) - \chi(Y,\calA)\\
&= \chi(Y,\calB) - \chi(Y,\calA).
\end{align*}

The first result now follows from \eqref{E:chirel} and the second from subtracting~\eqref{eq:localTerms} from~\eqref{eq:locsymEuler}.
\end{proof}

\section{Global symmetric differentials II}
\label{S:GlobalSymmetricDifferentials-II}

In this section we combine Theorem~\ref{thm:main} with the calculation in~\S\ref{S:local_euler_char} of the local Euler characteristic of symmetric differentials for $A_1$ singularities to obtain quasi-hyperbolicity results for surfaces of general type that are smooth except for finitely many isolated $A_1$ singularities.

Throughout this section, $X\subset \PP^n$ denotes a complex projective surface that is a complete intersection of multidegree $(d_1,\ldots,d_{n-2})$, with a singular locus $S$  consisting of $\ell$ isolated $A_1$ singularities $\{s_1,\ldots,s_\ell\}$. We let $\tau\colon Y\to X$ be a minimal resolution of $X$, with exceptional locus $E$, which consists of $\ell$ disjoint $(-2)$-curves $E_1,\dots E_\ell$, each isomorphic to $\PP^1$. We assume that $d_1 + \cdots + d_{n-2} > n+1$, so that $Y$ is of general type.

\begin{example} 
\label{ex:hypersurfaces}
Write $c_1(Y)^2=K^2$ and $c_2(Y)=\chi$.  Using~\eqref{E:chiYSm} and~\eqref{eq:locsymEuler}, inequality~\eqref{E:h0Yestimate_chi} gives
\[h^0(Y,S^m\Omega^1_Y)\geq\tfrac{1}{54}(9K^2-9\chi+8\ell)m^3-(\tfrac{1}{2}\chi-\tfrac{4}{9}\ell)m^2+O(m).\]
In particular, if $\ell > \frac{9}{8}(\chi-K^2)$ then the surface will have regular symmetric differentials for large enough $m$; in fact, $\liminf_{m\to \infty} h^0(Y,S^m\Omega^1_Y)/m^3 > 0$.
\end{example}

\begin{example}
\label{Ex:NodalInP3}
Let $X\subset \PP^3$ be a hypersurface of degree $d\geq 5$ with $\ell$ singularities of type $A_1$. By~\eqref{E:chiYSm}, we have
\[\chi(Y,S^m\Omega^1_Y)=-\tfrac{1}{3}(2d^2-5d)m^3-\tfrac{1}{2}(d^3-4d^2+6d)m^2-\tfrac{1}{6}(2d^3-10d^2+17d)m+\tfrac{1}{6}(d^3-6d^2+11d),\]
so \eqref{E:h0Yestimate_chi} together with \[\ell>\frac{9}{4}(2d^2-5d)\] implies $\liminf_{m\to \infty} h^0(Y,S^m\Omega^1_Y)/m^3 > 0$.
\end{example}

\begin{remark}
	\label{rem:BORR}
	We document here where our results differ from those stated in \cite{BO}.
	It has been previously noted (see \cite{RoulleauRousseau}*{Remark~12}) that \cite{BO}*{Lemma~2.2} is flawed. In particular, $\chi^0(s,S^m\Omega^1_Y)$ is overestimated in \cite{BO}*{(2.11)}, yielding an estimate for $\chi^1(s,S^m\Omega^1_Y)$ that is only quadratic in $m$. With our approach \eqref{E:h0Yestimate_chi}, no fixed value of $\ell$ would be sufficient to overcome the negative coefficient of $m^3$ in $\chi(Y,S^m\Omega^1_Y)$ that a nodal hypersurface $X$ would give rise to. Instead, \cite{BO}*{Theorem~2.6} uses a different approach where, via Serre duality, the authors establish the inequality
	\begin{equation}\label{E:h0Yestimate_serre_duality}h^0(Y,S^m\Omega^1_Y)\geq \chi(Y,S^m\Omega^1_Y)+\ell\chi^0(s,S^m\Omega^1_Y).\end{equation}
	Note the difference between this inequality and~\eqref{E:h0Yestimate_chi}: the latter uses $\chi^1(s,S^m\Omega^1_Y)$ in place of $\chi^0(s,S^m\Omega^1_Y)$.
	With the corrected asymptotic of $\chi^0(Y,S^m\Omega^1_Y)=\frac{11}{108}m^3+O(m^2)$, this gives that $\Omega^1_Y$ is big if
	\[\ell\geq \frac{36}{11}(2d^2-5d).\]
	This result is weaker than the one in Example~\ref{Ex:NodalInP3}.
	
	Roulleau and Rousseau establish analogues of \eqref{E:h0Yestimate_chi} and \eqref{E:h0Yestimate_serre_duality} for arbitrary $A_k$ singularities~\cite{RoulleauRousseau}*{Theorem~9} without proving an analogue of Proposition~\ref{P:codim_E_regular} for arbitrary $A_k$ singularities: instead, they get an asymptotic bound by cleverly taking the break-even point of the two approaches. For hypersurfaces with $A_1$ singularities they find the intermediate bound of 
	\[
	\ell > \frac{8}{3} (2d^2-5d).
	\]
	It would be interesting to see which of \eqref{E:h0Yestimate_chi} and \eqref{E:h0Yestimate_serre_duality} gives better results for varying $k$.
\end{remark}

\section{Computing regular differentials}
\label{S:comp_reg_diff}

In this section we describe how for $X\subset \PP^n$ explicitly given as a complete intersection $f_1=\cdots=f_{n-2}=0$, we can compute an explicit representation of
$\HH^0(X,(\tau_*S^m\Omega^1_Y)^{\vee\vee})$ and determine $\HH^0(Y,S^m\Omega^1_Y)$ as a subspace. We write $R=R_X=k[x_0,\ldots,x_n]/(f_1,\ldots,f_{n-2})$ for the projective coordinate ring and write $R(d)$ for the $R$-module obtained by shifting the grading such that $R(d)_i=R_{d+i}$.

An algebraic sheaf $\sF$ on $X$ determines a graded $R$-module
\[\Gamma_*(\sF)=\bigoplus_{d\in\Z} \HH^0(X,\sF\otimes\sO_X(d)).\]
In turn, the sheaf $\sF$ is determined by this graded module. Any graded $R$-module $M$ also determines an algebraic sheaf $\sF_M$ on $X$, and $\Gamma_*(\sF_M)$ is the \emph{saturation} of $M$. Since $X$ is a complete intersection that is nonsingular in codimension $1$, it is normal and projectively normal. This means that $R_X$ is saturated and hence that $\Gamma_*(\sO_X(d))=R(d)$.

We construct a module representing $\Omega_X^1$ in the following way.
Let $R_{\PP^n}=k[x_0,\ldots,x_n]$ be the projective coordinate ring of $\PP^n$. We have that
\[M_{\PP^n}=\Gamma_*(\Omega^1_{\PP^n})\subset R_{\PP^n}(-1)^{n+1}=\bigoplus_{i=0}^n R_{\PP^n}\,dx_i,\]
fits in an exact sequence
\[R_{\PP^n}(-3)^{\binom{n+1}{3}}\to R_{\PP^n}(-2)^{\binom{n+1}{2}}\to M_{\PP^n}\to 0,\]
where the second module has an $R$-basis $\{\omega_{ij}: 0\leq i < j \leq n\}$, the second map is given by $\omega_{ij}=x_idx_j-x_jdx_i$, and the relations are generated by the obvious
\[x_i\omega_{jk}-x_j\omega_{ik}+x_k\omega_{ij}\text{ for }0\leq i<j<k\leq n.\]
In order to compute a module for $\Omega^1_X$, we consider the submodule
\[\partial I_X:=\langle \partial_{x_i}(f_j) dx_i: j=1,\ldots,n-2;\; i=0,\ldots,n\rangle\subset \bigoplus_{i=0}^n R_{\PP^n}\,dx_i.\]
Then $(\partial I_X \cap M_{\PP^n})\otimes R_X$ yields the conormal sheaf on $X-S$, so the module 
\[M=M_X=M_{\PP^n}/(\partial I_X \cap M_{\PP^n})\otimes R_X\]
gives $\Omega^1_X$. We then construct $S^m M_X$ as an appropriate quotient of $M^{\otimes m}_X$.

Given a graded $R$-module $M$, we consider its dual $M^\vee=\Hom(M,R)$. 
On the level of sheaves, this corresponds (up to shift) to taking the sheaf hom $\sHom_{\sO_X}(\sF_M,\sO_X)$.
By applying this operation twice, we get a graded $R$-module with a homomorphism $M\to M^{\vee\vee}$. We write $\hat{S}^m\Omega^1_X=\sF_{(S^mM_X)^{\vee\vee}}$.

We emphasize here that all module operations used here can be performed by appropriate commutative algebra software such as Magma \cite{magma} and Macaulay2 \cite{M2}, using Gr\"obner bases. We have made code  available that implement the above ideas in the case of Barth's sextic and the perfect cuboid surface to interested readers in~\cite{electronic}.  See \S\S\ref{s:QuadricCIs}--\ref{s:nodalsurfaces} for more details.

\subsection{Computing an abstract presentation of $\hat{S}^m\Omega^1_X$}

The following lemma collects the results that relate a graded module to regular differentials.

\begin{lemma}\label{L:compute_sections}
	With the notation above, we have the following properties:
	\begin{enumerate}
		\item[(a)] $\hat{S}^m\Omega^1_X$ is reflexive,
		\item[(b)] $\hat{S}^m\Omega^1_X |_{X-S}=S^m\Omega^1_{X-S}$,
		\item[(c)] $\HH^0(Y- E,S^m\Omega^1_Y)\simeq \HH^0(X- S,S^m\Omega^1_X)\simeq \HH^0(X,\hat{S}^m\Omega^1_X)=(S^m M_X)^{\vee\vee}_0$.
	\end{enumerate}
    Let $H_X$ be a hyperplane section of $X$ and let $H_Y$ be the proper transform of $H_X$ on $Y$. Then
    \begin{enumerate}
    	\item[(d)] $\HH^0(Y-E,(S^m\Omega^1_Y)(-H_Y))\simeq \HH^0(X,(\hat{S}^m\Omega^1_X)(-H_X))\simeq(S^m M_X)^{\vee\vee}_{-1}$.
    \end{enumerate}
\end{lemma}
\begin{proof}
We have (a) because	the dual of a coherent sheaf on a normal variety is reflexive. Furthermore, $S^m\Omega^1_{X-S}$ is already reflexive, giving (b).

Since $X$ is a normal variety, sections of a reflexive sheaf on $X-S$ extend uniquely to $X$. Since $R_X$ is saturated, we have that $\Hom(M,R_X)$ is also saturated, so $\Gamma_*(\hat{S}^m\Omega^1_X)=(S^m M_X)^{\vee\vee}$, which proves (c).

Statement (d) is most easily argued with a hyperplane $H_X$ disjoint from $S$. Then $\tau$ induces $\HH^0(Y-E,(S^m\Omega^1_Y)(-H_Y))\simeq\HH^0(X-S,S^m\Omega^1_X(-H_X))$, so the first isomorphism follows from the same argument as for (c). The second holds because for any sheaf $\sF$ on $X$ we have $\sF(-H_X)\simeq \sF\otimes\sO_X(-1)$.
\end{proof}

\subsection{Representing global sections with K\"ahler differentials on $k(X)$}

We explain how an abstract representation of an element in $\HH^0(X,\hat{S}^m\Omega^1_X)$ can be turned into a recognizable representation of an element of $\HH^0(X- S,{S}^m\Omega_X^1)$.

Remember that we have a representation of $M=S^m M_X$ as a quotient of $R(-2m)^{r_M}$, our generators being the monomials of degree $m$ in $\omega_{ij}=x_idx_j-x_jdx_i$.

We, or rather a computer algebra system, compute $M^\vee=\Hom_R(M,R)$ as a quotient of free modules, defined by an exact sequence
\[K_{M^\vee}\to\bigoplus_{i=1}^{r_{M^\vee}} R(d_i')\to M^\vee\to 0.\]
The bilinear pairing $M\times M^\vee\to R$ is given by an $r_M\times r_{M^\vee}$ matrix $A$ over $R$. The hard work, accomplished by Gr\"obner basis computations, consists of determining the correct one.

Similarly, we obtain a description of $M^{\vee\vee}=\Hom_R(M^\vee,R)$ as a quotient defined by 
\[K_{M^{\vee\vee}}\to\bigoplus_{i=1}^{r_{M^{\vee\vee}}} R(d_i)\to M^{\vee\vee}\to 0\]
together with an $r_{M^\vee}\times r_{M^{\vee\vee}}$ matrix $B$ over $R$, describing the pairing $M^\vee\times M^{\vee\vee}\to R$.

In order to get a recognizable representation of our symmetric differentials, we evaluate them at the generic point. Say, we take the affine open $X- \{x_0=0\}$. The dehomogenization map $R\to k[X- \{x_0=0\}]$ corresponding to $(x_0:\cdots:x_n)=(1:x_1:\cdots:x_n)$ gives us a module $M^\text{aff}$, and we know that $M^\text{aff}\otimes k(X)$ gives us an $(m+1)$-dimensional $k(X)$-vector space with for instance the basis $\calB=\{dx_1^{m-i}dx_2^i: i =0,\ldots,m\}$. Note that $\omega_{10}=x_1dx_0-x_0dx_1$ equals $-dx_1$ if $x_0=1$, so we can readily recognize this basis from the generators we have chosen for $M$.

We know that $$(M^{\vee\vee})^\text{aff}\otimes k(X)$$ is isomorphic to 
this vector space. We take the submatrix $A'$ of $A$ consisting of the $m+1$ rows that correspond to the basis $\calB$. We know that $A'$ has rank $m+1$ over $k(X)$, so we select $m+1$ columns of $A'$ to get a square submatrix $A''$ that is invertible over $k(X)$.

Let $B'$ be the submatrix of $B$ obtained by taking the $m+1$ rows matching the columns chosen for $A''$. Then 
\[(dx_1^m,dx_1^{m-1}dx_2,\ldots,dx_2^m)(A'')^{-1}(B')^T\]
gives expressions for the generators of $M^{\vee\vee}$ as K\"ahler differentials on $k(X)$.

\subsection{Determining the conditions to extend into the exceptional locus on $Y$}
\label{s:conditions_to_extend_into_E}
Let $\omega\in \HH^0(X,\hat{S}^m\Omega^1_X)$ and let $s\in S$ be an $A_1$ singularity. Without loss of generality, we may assume that $x_1,\ldots,x_n$ provide an affine chart around $s$, and that the tangent space of $X$ at $s$ is $x_4=\cdots=x_n=0$, and that the tangent cone of $X$ at $s$ inside the tangent space is defined by $x_1x_3=x_2^2$. Let $E_s$ be the exceptional curve on $Y$ above $s$. Then the completed local ring of $Y$ at $E_S$ is isomorphic to $k(u)[[t]]$ and we have
\[(x_1,x_2,x_3,x_4,\ldots,x_n)=(t,tu,tu^2,0,\ldots,0)\pmod{t^2},\]
This allows us to compute an expansion
\[\omega=\sum_{i=0}^n a_i(u,t)dt^idu^{n-i},\]
where $a_i(u,t)\in k[u]((t))$. Over $k$, there will be only finitely monomials $t^au^bdt^idu^{n-i}$ occurring with $a<0$, so we get a finite system of equations on $\HH^0(X- S,S^m\Omega^1_X)$ to be satisfied for an element to extend to a regular form along $E_s$ on $Y$. In fact, Proposition~\ref{P:codim_E_regular} gives us an upper bound on the number of equations we get.

\section{Concluding quasi-hyperbolicity and explicitly computing the locus of special curves}
\label{s:resultants}

In this section we consider the situation of Proposition~\ref{C:regdif_lowerbound_quasiproj}. We take $X$ to be a complete intersection with singular locus $S=\{s_1,\ldots,s_\ell\}$ consisting of $A_1$ singularities, with $s_1,\ldots,s_r$ removed. We take $Y$ to be a minimal resolution of $X$, so $Y$ has the exceptional curves $E_1,\ldots,E_r$ removed.

A regular symmetric differential on $Y$ restricts \emph{complete} curves on $Y$, so we can obtain information on curves of genus $0$ and $1$ on $X$ that avoid the singularities $s_1,\ldots,s_r$.

The existence of regular symmetric differentials is usually concluded by observing that the lower bounds in Proposition~\ref{C:regdif_lowerbound_quasiproj} are cubic in $m$. The lemma below implies that regular differentials that vanish along a divisor similarly exist in that case.

\begin{lemma}\label{L:also_hyperplane_vanishing}
Suppose that $Y$ is a quasi-projective nonsingular surface of general type and suppose that there is a constant $c>0$ such that
	\[h^0(Y,S^m\Omega^1_Y)=cm^3+O(m^2).\]
Suppose that $H$ is a divisor on $Y$ (for instance, if $Y$ is  a minimal resolution of an $X$ as above, we can take $H$ to be the inverse image of a general hyperplane section of $X$). Then
\[h^0(Y,(S^m\Omega^1_Y)(-H))>0 \text{ for large enough } m\]
as well.
\end{lemma}
\begin{proof}
The condition amounts to the assertion that $\Omega^1_Y$ is \emph{big}. It is a standard result that for a \emph{big} bundle $\sE$ and a line bundle $\sL$, we have that $h^0(Y,S^m\sE \otimes \sL)>0$ for large enough $m$.
\end{proof}

We use that on a complete curve $C\subset Y$ of genus $0$ we have $\HH^0(C,S^m\Omega^1_C)=0$, and on a curve of genus $1$, any section that vanishes somewhere must be identically $0$. If we have $\omega\in \HH^0(Y,S^m\Omega^1_Y)$ then we see that $\omega$ pulled back to $C$ must be identically $0$ if $C$ is of genus $0$. Similarly, if for an effective divisor $H$ intersecting $C$ we have $\omega\in\HH^0(Y,(S^m\Omega^1_Y)(-H))$ then 
$\omega$ restricts to a regular differential on $C$ that vanishes somewhere, so if $C$ is of genus $0,1$, then $\omega$ must restrict to $0$ on $C$. The proposition below recalls how the foliation determined by $\omega$ can be used to establish quasi-hyperbolicity.

\begin{prop}\label{P:foliation_quasihyperbolicity} Suppose $Y$ is as above and that $\omega\in\HH^0(Y,S^m\Omega^1_Y)$. Then there are only finitely many complete curves $C$ on $Y$ of genus at most $1$ on which $\omega$ restricts to $0$.
\end{prop}
\begin{proof}[Sketch of proof.]  (See \cite{Debarre2004} for more details.) The form $\omega$ defines a degree $m$ form on the projective bundle $\PP^1(\Omega^1_Y)$, and therefore gives rise on a surface $Y'$ covering $Y$. On a desingularization $\tilde{Y}$ of $Y'$, the form $\omega$ induces a foliation (formed by the integral curves defined by the degree $m$, first order differential equation that $\omega$ describes on $Y$. By the observation above, any curve $C$ as above would be a leaf of this foliation. By
Jouanolou~\cite{Jouanolou1978}, such a foliation either contains only finitely many algebraic leaves, or the foliation is in fact an algebraic fibration of $\tilde{Y}$ over a curve. Since $\tilde{Y}$ is still of general type, the general member of such a fibration must be of genus larger than $1$, and therefore contain only finitely many fibres of genus $0$ or $1$. In either case, the result follows.
\end{proof}

In special cases, the foliation induced by $\omega$ can be determined explicitly, but in general this seems to be hard, since it essentially requires solving a first order differential equation of degree $m$ on $Y$. We sketch another computational method here, that uses two sections $\omega_1,\omega_2$, and determines a closed locus in $Y$ that contains all curves to which $\omega_1,\omega_2$ pull back to $0$. We determine conditions for points $P\in Y$ such that there can be a curve $C$ through $P$ on $Y$ on which both $\omega_1,\omega_2$ pull back to $0$.

At a point $P\in Y$, a form $\omega$ defines a homogeneous degree $m$ form on the tangent space $T_P(Y)$. If $C$ is nonsingular at $P$, then the kernel of $T_P(Y)^*\to T_P(C)^*$ is generated by a single element, and $\omega(P)$ must be divisible by it. More specifically, if  $x,y$ are affine coordinates on $Y$ such that $dx(P),dy(P)$ span $T_P(Y)^*$
then we have that $\omega(P)=\sum_{i=0}^{m} a_i dx^idy^{m-1}$, for $a_i\in k(P)$ and that the kernel of $T_P(Y)^*\to T_P(C)^*$ is spanned by an element $\alpha_0 dx-\alpha_1 dy$. We need that $\sum_{i=0}^{m} a_i \alpha_1^i\alpha_0^{m-1}=0$, i.e., that $(\alpha_0:\alpha_1)$ is a root of $\omega$ as a form on $\PP^1(k(P))$.

Suppose now that we have two such forms $\omega_1,\omega_2$; say $\omega_1(P)=\sum_{i=0}^{m-i} a_i(P) dx^idy^{m-1}$ and $\omega_2(P)=\sum_{i=0}^{m-i} b_i(P) dx^idy^{m-1}$. Then for $P$ to be a nonsingular point on $C$, we need that $\omega_1(P)$ and $\omega_2(P)$, as forms on $\PP^1(k(P))$, have a root in common, i.e., have a vanishing resultant. Provided that this resultant is not identically $0$ on $Y$, we get a proper closed subset that contains any such $C$ (since all the points of $C$ lie in the closure of its nonsingular points).

We define the locus $\res_{x,y}(\omega_1,\omega_2)$ to be the locus where $dx,dy$ do not span $T_P(Y)^*$ or where the following Sylvester determinant vanishes
\[\res_{x,y}(\omega_1,\omega_2)=\det\begin{pmatrix}
a_0(P)&\cdots&a_m(P)\\
&\ddots&&\ddots\\
&&a_0(P)&\cdots&a_m(P)\\
b_0(P)&\cdots&b_m(P)\\
&\ddots&&\ddots\\
&&b_0(P)&\cdots&b_m(P)\\
\end{pmatrix},\]
and let $\res(\omega_1,\omega_2)$ be the intersection of the vanishing of $\res_{x,y}(\omega_1,\omega_2)$ for all possible choices of $x,y$ (it is sufficient to use all standard affine coordinate pairs derived from a nonsingular quasi-projective model of $Y$).

\begin{prop}\label{P:explicit_resultant_locus_g0}
	Let $Y$ be as above and suppose that $\omega_1,\omega_2\in \HH^0(Y,S^m\Omega^1_Y)$. Then any complete genus $0$ curve $C\subset Y$ is contained in $\res(\omega_1,\omega_2)$.
\end{prop}

\begin{proof}
As explained above, $\omega_1,\omega_2$ pull back to regular symmetric differentials on $C$. If $C$ is a complete curve of genus $0$, this means they pull back to $0$. By the discussion above, this implies that $C\subset \res(\omega_1,\omega_2)$.
\end{proof}

\begin{prop}\label{P:explicit_resultant_locus_g1}
	Let $Y$ be as above and suppose $H\subset Y$ is an effective divisor on $Y$. If $\omega_1,\omega_2\in \HH^0(Y,S^m\Omega^1_Y(-H))$, then any complete genus $1$ curve $C$ on $Y$ that intersects $H$ is contained in $\res(\omega_1,\omega_2)$.
\end{prop}
\begin{proof}
As explained above, $\omega_1,\omega_2$ pull back to regular symmetric differentials on $C$. Furthermore, $C\cap H$ yields zeros of $\omega_1,\omega_2$, so they pull back to $0$.
\end{proof}

\begin{cor}\label{C:explicit_resultant_locus}
	Let $X\subset\PP^n$ be a complete intersection surface with a singular locus $S=\{s_1,\ldots,s_\ell\}$ consisting of $A_1$ singularities. Suppose that for a hyperplane section $H_X$ of $X$ we have $\omega_1,\omega_2\in \HH^0(X,\hat{S}^m\Omega^1_X(-\floor{\frac{m}{2}} H_X))$. Then any genus $0$ curve $C$ on $X$
	\begin{enumerate}
		\item is contained in $\res(\omega_1,\omega_2)$, or
		\item is contained in one of the finitely many explicitly determinable hyperplanes, or
		\item passes through at least $n+1$ distinct singularities, because $C\cap S$ spans $\PP^n$.
	\end{enumerate}
Furthermore, any genus $1$ curve $C$ on $X$
	\begin{enumerate}
		\setcounter{enumi}{3}
		\item is contained in $\res(\omega_1,\omega_2)$, or
		\item is contained in a linear subspace of dimension at most one more than the span of $C\cap S$. In particular, $C$ passes through at least two singularities and if $C\cap S$ spans a linear space of dimension at most $n-2$, then the degree of $C$ is at most the degree of $X$.
	\end{enumerate}
\end{cor}

\begin{proof}
We first note that hyperplane sections are linearly equivalent, so if $H'$ is another hyperplane section then there is a function $f\in \HH^0(X,\sO_X(\floor{\frac{m}{s}}(H_X-H'))$ such that $f\omega_1,f\omega_2\in \HH^0(X,\hat{S}^m\Omega^1_X(-\floor{\frac{m}{2}} H')$. Then $\res(\omega_1,\omega_2)$ and $\res(f\omega_1,f\omega_2)$ will only differ by components contained in $H_X$ and $H'$.

Let us first deal with $C$ of genus $0$. Suppose we are not in case (3), so $C\cap S$ does not span $\PP^n$.  We can choose a hyperplane $H'$ that contains the span. Since there are only finitely many possibilities for $C\cap S$, we can choose $H'$ from a finite collection. For instance, if $S$ as a whole spans $\PP^n$, it is sufficient to consider all hyperplanes spanned by singularities.

Setting $S'=S-(S\cap H')$, we see that $C$ is a complete curve in $X-S'$, so application of Proposition~\ref{P:explicit_resultant_locus_g0} yields that the proper transform of $C$ to $Y$ lies in $\res(\pi^* f\omega_1,\pi^* f\omega_2)$, so $C\in\res(f\omega_1,f\omega_2)$. Since $\res(\omega_1,\omega_2)$ and $\res(f\omega_1,f\omega_2)$ differ by a predetermined set of hyperplane sections, we see that we are in case (1) or (2).

For a genus $1$ curve $C$ for which $C\cap S$ spans a linear space of dimension at most $n-2$ then we can choose a point $P$ on $C$ that is not in $S$ and consider a hyperplane $H'$ that contains $P$ and $C\cap S$. By the same argument as above, we see that $C$ must lie in $\res(f\omega_1,f\omega_2)$. 

If the codimension of the linear span of $C\cap S$ inside the linear span of $C$ is at least $2$, we can choose our hyperplane so that $C$ is not contained in it, forcing $C$ to lie in $\res(\omega_1,\omega_2)$. If $C\cap S$ consists of just one point, then this is surely the case, since a genus $1$ curve is not a line.

If the span of $C\cap S$ is of dimension at most $n-2$ and $C$ spans a space that is at most one dimension more, then $C$ is a component of a hyperplane section of $X$. This bounds its degree.
\end{proof}

\begin{rmk}
Corollary~\ref{C:explicit_resultant_locus}, case (5) is perhaps a little disappointing, but it still accomplishes a significant reduction: any genus $1$ curve not in $\res(\omega_1,\omega_2)$ that passes through at most $n-1$ singularities must be a component of a hyperplane section. 
The space of hyperplane sections is finite-dimensional and for a section to have a genus $1$ component it must be highly singular or reducible. These conditions define $0$-dimensional loci in the space of hyperplane sections of $X$, which can, at least in principle, be determined explicitly. 
\end{rmk}

\section{Applications to complete intersections of quadrics}
\label{s:QuadricCIs}

In this section we consider surfaces $X\subset \PP^n$ that are complete intersections of $n-2$ quadratic equations, with $\ell$ isolated $A_1$ singularities. In this case we have
\[K^2=c_1(Y)^2=(n-5)^2 2^{n-2},\; \chi=c_2(Y)=(n^2-7n+16)2^{n-3},\]
and that $Y$ is of general type if $n\geq 6$ by Lemma~\ref{L:Y_general_type}. By~\eqref{E:chiYSm} we have
\[\chi(Y,S^m\Omega^1_Y)=\frac{1}{3}2^{n-5}\big(
2(n^2-13n+34)m^3-6(n^2-7n+16)m^2-(5n^2-41n+98)m+3n^2-27+66
\big).\] 
Theorem~\ref{thm:main} gives that for $\ell \geq \ell_{\mathrm{min}}(n)$, we have regular symmetric differentials on $Y$ for sufficiently large $m$, where $\ell_{\mathrm{min}}(n)$ is defined by
\[
\begin{array}{c||c|c|c|c|c}
n&6&7&8&9&\geq 10\\
\hline
\ell_\mathrm{min}(n)&73&145&217&145&0\\
\end{array}\]
 In fact, using Lemma~\ref{L:also_hyperplane_vanishing}, we also have differentials vanishing along an ample divisor. Hence, using Proposition~\ref{P:foliation_quasihyperbolicity}, we see that such $Y$ are algebraically quasihyperbolic.

One concrete example is Theorem~\ref{thm:magicsquares} on the surface $\Xms\subset\PP^8$ with $\ell=256$ singularities of type $A_1$. Since $\ell > \ell_\mathrm{min}(8)=217$, the surface is algebraically quasihyperbolic. In fact, using \eqref{E:h0Yestimate_chi} we find that for $m\geq 47$ there are global sections and that $\HH^0(\Yms,\Omega^1_{\Yms})\geq 8448$. Unfortunately, $\Xms$ is out of range of current computational techniques to explicitly determine $\hat{S}^m\Omega^1_{\Xms}$, so we cannot apply the methods from Corollary~\ref{C:explicit_resultant_locus} get an explicit description of the locus of special curves.

As a computationally more accessible example, let us consider the projective surface $X=\Xpc$ that parametrizes \emph{perfect cuboids}, i.e., bricks with all sides $x_1,x_2,x_3$, diagonals $y_1,y_2,y_3$, and body diagonal $z$ rational. The surface is a complete intersection in $\PP^6$, described by the quadratic equations
\[X\colon
\left\{
\begin{aligned}
y_1^2&=x_2^2+x_3^2,\\
y_2^2&=x_3^2+x_1^2,\\
y_3^2&=x_1^2+x_2^2,\\
z^2&=x_1^2+x_2^2+x_3^2.
\end{aligned}\right.
\]
Its singular locus $S$ consists of $\ell=48$ singularities of type $A_1$, so algebraic hyperbolicity does not follow immediately for its minimal desingularization $Y$. However, applying Proposition~\ref{C:regdif_lowerbound_quasiproj} with $r=48-13$, we find that
\[h^0(Y-(E_1\cup\cdots\cup E_r),S^m\Omega^1_Y)=\frac{1}{108}m^3+O(m^2)\]
and hence that there are only finitely many curves of genus $0$ or $1$ on $X$ that pass through at most $13$ singularities. The lower bound based on Euler characteristics only turns positive for $m\geq 862$; a value not within range for explicit computation.

For $m=2$, we find via explicit computation (see \cite{electronic}) that $h^0(X,\hat{S}^2\Omega^1_X)=13$, with generators as listed in Table~\ref{tbl:sym2gens}. An indication that the $\omega_i$ are regular on $X-S$ is that the denominators listed are supported on $y_1y_2y_3z=0$, which is the branch locus of the projection on $(x_1:x_2:x_3)$.

Note that $\chi(X,\hat{S}^2\Omega^1_X)=7$, so even with the assumption that $h^2(X,\hat{S}^2\Omega^1_X)=0$, the Euler characteristic underestimates the dimension of the space of global sections. Furthermore, $\omega_7$ vanishes along $H\colon x_1=0$ and $\langle\omega_7\rangle=\HH^0(X,(\hat{S}^2\Omega^1_X)(-H))$. This means that $\hat{S}^2\Omega^1_X$ admits sections vanishing along any hyperplane. 
\begin{table}
\[\begin{aligned}
\omega_1&=\frac{x_2x_3}{y_3^2z^2} (dx_2)^2-\frac{2}{z^2}dx_2dx_3+\frac{x_2x_3}{y_2^2z^2}(dx_3)^2\\
\omega_2&=\frac{x_3(y_1^2+y_3^2)}{y_1^2y_3^2z^2}(dx_2)^2-\frac{2x_2}{y_1^2z^2}dx_2dx_3-\frac{x_3}{y_1^2z^2}(dx_3)^2\\
\omega_3&=\frac{x_2(x_3^2-1)}{y_1^2y_3^2z}(dx_2)^2-\frac{2x_3}{y_1^2z}dx_2dx_3+\frac{x_2}{y_1^2z}(dx_3)^2\\
\omega_4&=\frac{1}{y_3^2z}(dx_2)^2-\frac{1}{y_2^2z}(dx_3)^2\\
\omega_5&=\frac{x_2}{y_1^2z^2}(dx_2)^2+\frac{2x_3}{y_1^2z^2}dx_2dx_3-\frac{x_2(y_1^2+y_2^2)}{y_1^2y_2^2z^2}(dx_3)^3\\
\omega_6&=\frac{x_3}{y_1^2z}(dx_2)^2-\frac{2x_2}{y_1^2z}dx_2dx_3+\frac{x_3(x_2^2-1)}{y_1^2y_2^2z}(dx_3)^2\\
\omega_7&=\frac{1}{y_1y_2y_3z^2}\Big((x_3^2+1)(dx_2)^2-2x_2x_3dx_2dx_3+(x_2^2+1)(dx_3)^2\Big)\\
&x_2\omega_7,x_3\omega_7,y_1\omega_7,y_2\omega_7,y_3\omega_7,z\omega_7
\end{aligned}\]
\caption{Generators for $\HH^0(X,\hat{S}^2\Omega^1_X)$; given on affine patch $x_1=1$}
\label{tbl:sym2gens}
\end{table}

As we show below, the foliation determined by $\omega_7$ can be described sufficiently explicitly to obtain stronger results, but first we sketch how Proposition~\ref{P:explicit_resultant_locus_g0} can be used to obtain information on the genus $0$ curves on $X$ without solving differential equations. One can check via the approach in Section~\ref{s:conditions_to_extend_into_E} that $\pi^*\omega_1$ is regular along the exceptional curves on $Y$ over the singularities with $y_1=0$. We can then compute $\res(\omega_1,y_1\omega_7)$ to conclude that $X$ contains no genus $0$ curves that pass only through singularities for which $y_1=0$. By symmetry, the same holds for $y_2=0$ and $y_3=0$, and since every node on $X$ satisfies $y_1y_2y_3=0$, we see that any genus $0$ curve has to pass through at least two distinct nodes.

For the proof of Theorem~\ref{thm:CuboidIntro} we need some information on the curves that do lie on $X$. The list of curves in the lemma below already appears in \cite{vanLuijk2000}.
\begin{lemma}\label{L:cuboid_hyperplanecurves}
Suppose $L\subset X$ is a curve of genus at most $1$, contained in a hyperplane $H$ spanned by nodes of $X$. Then $L$ is one of the following curves.
\begin{itemize}
	\item $8$ genus $0$ curves satisfying $x_1^2+x_2^2+x_3^2=0$, defined over $\Q(i)$,
	\item $24$ genus $0$ curves satisfying $x_1x_2x_3=0$, defined over $\Q(i)$,
	\item $24$ genus $1$ curves satisfying one of three equations of the form $2x_1^2+x_2^2+x_3^2=0$, defined over $\Q(i,\sqrt{2})$, each through three non-collinear singularities of $X$,
	\item $36$ genus $1$ curves satisfying one of three equations of the form $x_1^4-x_2^4$, defined over $\Q(i)$ or $\Q(\sqrt{2})$, each through three or four non-collinear singularities of $X$.
\end{itemize}
\end{lemma}
\begin{proof}
The singular locus $S$ consists of $48$ points, so there are at most $\binom{48}{6}$ hyperplanes $H$ to be considered. As it turns out, there are somewhat less than $60000$ of them, forming $2442$ orbits under the $384$ obvious linear automorphisms generated by the simultaneous permutation action on $x_1,x_2,x_3$ and $y_1,y_2,y_3$ and the sign changes on each variable. We establish the lemma by considering representatives of each orbit, decomposing $H\cap X$, and checking which components are curves of genus at most $1$. See \cite{electronic} for a transcript of the computations. We find the list stated. Note that all the curves are nonsingular, that the genus $0$ curves we find are plane conics and that the genus $1$ curves we find are complete intersections of quadrics in $\PP^3$, each through at least three, non collinear singularities.
\end{proof}

\begin{lemma}\label{L:integral_curve_classification_on_P2}
	Let $\eta$ be the degree two symmetric differential form on $\PP^2$ that on the affine patch $(1:x_2:x_3)$ is given by $(x_3^2+1)(dx_2)^2-2x_2x_3dx_2dx_3+(x_2^2+1)(dx_3)^2$. The integral curves of $\eta$ (that is to say, curves in $\PP^2$ onto which $\eta$ pulls back to an identically vanishing symmetric differential form) are the conic $x_1^2+x_2^2+x_3^2=0$ and the tangent lines to it, given by $Ax_1+Bx_2+Cx_3=0$ with $A^2+B^2+C^2=0$.
\end{lemma}
\begin{proof}
	It is straightforward to check that the given curves are indeed integral curves for $\eta$.
	
	For instance, for $C\neq 0$ we use the parametrization $(x_1:x_2:x_3)=(1:t:(-Bt-A)/C)$. On that line we have $dx_2=dt$ and $dx_3=-(B/C)dt$. Substitution into $\eta$ yields $(A^2+B^2+C^2)dt^2/C^2$.
	For the conic we check similarly through the parametrization $(x_1:x_2:x_3)=(\sqrt{-1}:(1-t^2)/(t^2+1):-2t/(t^2+1))$.

	Any point not on $x_1^2+x_2^2+x_3^2=0$ has exactly two tangent lines to the conic passing through it. Since $\eta$ is of degree $2$, an integral curve passing through a point $P$ must have one of at most two tangent directions. It follows that an integral curve to $\eta$ that is nonsingular at a point $P$ outside $x_1^2+x_2^2+x_3^2=0$ must be one of the tangent lines locally and therefore globally. This is sufficient to establish the lemma.
\end{proof}

\begin{proof}[Proof of Theorem~\ref{thm:CuboidIntro}]
Suppose that $L\subset X$ is a genus $0$ curve such that the singularities of $X$ it passes through are contained in a hyperplane $H$.
Let $h$ be the linear form defining $H$. Then $\omega=h\omega_7\in \HH^0(X,\hat{S}^2\Omega^1_X)$ vanishes along $H$. 
By Corollary~\ref{C:vanishing_on_hyperplane}, for any singularity $s$ of $X$ in $H$, we have that $\pi^*\omega$ is regular on the exceptional curve $E_s\subset Y$. Hence, we see that $\omega$ pulls back to $0$ on $L$.

We observe that $\phi\colon X\to\PP^2$ given by $(x_1:x_2:x_3)$ expresses $X$ as a finite, multiquadratic cover of $\PP^2$ of degree $16$, ramified over $(x_1^2+x_2^2)(x_1^2+x_3^2)(x_2^2+x_3^2)(x_1^2+x_2^2+x_3^2)=0$. For
\[\eta=(x_3^2+1)(dx_2)^2-2x_2x_3dx_2dx_3+(x_2^2+1)(dx_3)^2\]
and $h=x_1$, we see that $\omega=\phi^*(\eta)/(y_1y_2y_3z^2)$. It follows that $L$ must lie in $H$ or that $\phi(L)$ is a solution curve to $\eta$.
Lemma~\ref{L:cuboid_hyperplanecurves} lists the curves contained a hyperplane $H$ spanned by singularities.

The alternative is that $\phi\colon L\to\phi(L)$ expresses $L$ as a cover of one of the curves classified by Lemma~\ref{L:integral_curve_classification_on_P2}. The curves that cover $x_1^2+x_2^2+x_3^2=0$ are contained in the hyperplane $z=0$ and are included in Lemma~~\ref{L:cuboid_hyperplanecurves}. Therefore, let us assume that $\phi(L)$ is given by $Ax_1+Bx_2+Cx_3=0$, with $A^2+B^2+C^2=0$. Note that $X$ is a compositum
\[\xymatrix{
&X\ar[dr]^{(x_1:x_2:x_3:z)}\ar[dl]_{(x_1:\cdots:y_3)}\ar[dd]^\phi\\
X_y\ar[dr]&&X_z\ar[dl]\\
&\PP^2
}\]
and that $X_z\to\PP^2$ is ramified over $x_1^2+x_2^2+x_3^2=0$. Since $\phi(L)$ is tangent to this locus, we see that $\phi(L)$ pulls back to two components on $X_z$. 
On $X_y$, generically $\phi(L)$ pulls back to a nonsingular complete intersection of quadrics in $\PP^4\subset\PP^5$, so is a canonical genus $5$ curve. It follows in those cases that $L$ itself is isomorphic to this genus $5$ curve, contradicting that $L$ has genus $0$. Riemann-Hurwitz shows that this can only be avoided if $\phi(L)$ passes through a singular point of the branch locus of $X_y\to \PP^2$. However, that branch locus consists of tangent lines to $x_1^2+x_2^2+x_3^2=0$, so this only happens if $\phi(L)$ is one of the components of $(x_1^2+x_2^2)(x_1^2+x_3^2)(x_2^2+x_3^2)=0$. But such curves $L$ are contained in a hyperplane (such as $x_1+ix_2=0$). These are included in the list of curves in Lemma~\ref{L:cuboid_hyperplanecurves}, and these give genus $1$ curves.

Next we show that $X$ does not contain genus $1$ curves $C$ for which $X\cap S$ generates a linear space of codimension at least $2$ in the linear space generated by $C$. We argue by contradiction and assume $C$ is such a curve. Then we can choose a point $P$ on $C$ outside of the singular locus on $X$, such that $C$ is not contained in the linear span on $C\cap S$ and $P$. That means we can choose a hyperplane $H$ that contains $C\cap S$ and $P$, but not $C$ entirely. Let $h$ be the linear form defining $H$. By Corollary~\ref{C:vanishing_on_hyperplane}, the differential $h\omega_7$ pulls back to a regular one with a zero at $P$ on $C$, so it must pull back to $0$. However, we have constructed $H$ to not contain all of $C$, so $\phi\colon C\to \phi(C)$ expresses $C$ as a cover of one of the curves classified by Lemma~\ref{L:integral_curve_classification_on_P2}. As mentioned above, we do find some genus $1$ curves, but these are of degree $4$ and $C\cap S$ is easily checked to be of codimension at most $1$ in the linear space generated by $C$.

As special cases, note that a genus $1$ curve cannot be contained in a $1$-dimensional linear space, so any genus $1$ curve on $X$ must pass through at least two singularities.

Furthermore. if $C\cap S$ consists of at most $5$ points, we see that $C$ must be contained in a hyperplane section, limiting the degree of $C$ to $16$.
\end{proof}

\section{Applications to nodal surfaces in $\PP^3$}
\label{s:nodalsurfaces}

Let $X\subset \PP^3$ be a hypersurface of degree $d\geq 5$ with $\ell$ singularities of type $A_1$ and let $Y$ be its minimal resolution. We saw in Example~\ref{Ex:NodalInP3} that for $\ell>\frac{9}{4}(2d^2-5d)$ and sufficiently large $m$ we have $h^0(Y,S^m\Omega^1_Y) >0$. There are a few well-known surfaces of low degree $d$ and with many $A_1$ singularities.

\begin{center}
\begin{tabular}{l|c|c}
	name&$d$&$\ell$\\
	\hline
	Barth's sextic surface \cite{Barth1996}&6&65\\
	Barth's decic surface \cite{Barth1996}&10&345\\
	Sarti's surface \cite{Sarti2001}&12&600
\end{tabular}
\end{center}

Lemma~\ref{L:also_hyperplane_vanishing} allows us to conclude that Barth's decic surface and Sarti's surface are algebraically quasihyperbolic. In fact, from our lower bounds we find the lowest $m$ for which there are guaranteed to be global sections. For example, for Barth's decic surface we find that 
\[h^0(Y,S^m\Omega^1_Y)\geq \begin{cases}
\frac{2}{9}m^3-\frac{538}{3}m^2-82m+85&\text{ for } m\equiv 0\pmod3\\
\frac{2}{9}m^3-\frac{538}{3}m^2-82m+\frac{991}{9}&\text{ for } m\equiv 1\pmod3\\
\frac{2}{9}m^3-\frac{538}{3}m^2-\frac{472}{3}m+\frac{200}{9}&\text{ for } m\equiv 2\pmod3,
\end{cases}\]
These bounds turn positive when $m\geq160$, where we find $h^0(Y,S^{160}\Omega^1_Y)\geq 15755$. For Sarti's surface, a similar computation shows that the bounds turn positive when $m\geq 28$ and that $h^0(Y,S^{28}\Omega^1_Y)\geq 7646$. Neither of these values is within the range of practical computation to explicitly determine the locus of rational and genus 1 curves.

For Barth's decic surface $X=X_{10}$, Magma~2.24-6~\cite{magma} is just about capable of computing the graded module representing $S^2(\Omega^1_X)^{\vee\vee}$ over a finite field. We did so over $\F_{10009}$ and $\F_{50021}$. In both cases this took about 8 hours of computations on an Intel Xeon CPU E5-2660, 2.20GHz, using most of the 64Gb main memory. We find in both cases that $h^0(X,\hat{S}^2\Omega^1_X (-H))=7$ and that $h^0(X,\hat{S}^2\Omega^1_X)=7\cdot4+27=55$. We expect that these results are representative of what happens in characteristic $0$, which one could confirm by rational reconstruction (see below for a case where we actually executed this procedure). With that in place, the following results would be within reach.
\begin{itemize}
	\item By Proposition~\ref{P:codim_E_regular} and Proposition~\ref{C:regdif_lowerbound_quasiproj} we can choose any $17$ singularities and find $55-3\cdot17=4$ differentials that extend to the exceptional components above them, by the method explained in Section~\ref{s:conditions_to_extend_into_E} (and we find we cannot extend them into all $345$ components). Hence, Proposition~\ref{P:explicit_resultant_locus_g0} would give an approach to determining all genus $0$ curves on $X$ that pass through at most $17$ singularities, at the cost of some significant combinatorics.
	\item We can apply Corollary~\ref{C:explicit_resultant_locus} to find an explicit description of the genus $0$ curves $C$ on $X$ for which $C\cap X$ does not span all of $\PP^3$, in particular, all curves passing through at most $3$ singularities, as well as all genus $1$ curves that pass through at most $2$ singularities.
\end{itemize}
We do not pursue these particular results, but instead demonstrate similar results on the sextic surface $X_6$, as described below.

\subsection{Barth's sextic surface}
\label{ss:BarthSextic}
As an illustration, we perform similar computations for Barth's sextic surface $X_6$, defined over $\Q(\sqrt{5})$, with $\phi=\frac{1}{2}(\sqrt{5}+1)$ by
\[
X_6\colon 4(\phi^2x^2-y^2)(\phi^2y^2-z^2)(\phi^2z^2-x^2)-(1+2\phi)(x^2+y^2+z^2-w^2)^2=0.
\]
We find the following genus $0$ curves on $X_6$:
\begin{itemize}
	\item $6$ degree $1$ curves, each through $5$ singularities,
	\item $6$ degree $2$ curve, each through $10$ singularities,
	\item $15$ degree $6$ plane curves, each through $10$ singularities.
\end{itemize}
In addition we find the following genus $1$ curves on $X_6$:
\begin{itemize}
	\item $20$ degree $3$ plane curves, each through $15$ singularities,
	\item $10$ pairs of degree $3$ plane curves through $9$ singularities, defined over $\Q(\sqrt{5},i)$,
	\item $48$ degree $5$ plane curves, each through $10$ singularities,
	\item $15$ degree $4$, non-planar curves, each through $16$ singularities.
\end{itemize}
This list includes all genus $0,1$ curves on $X_6$ that lie in planes spanned by nodes.  

\begin{proof}[Proof of Theorem~\ref{thm:BarthIntro}]
As it turns out, computing the graded module representing
$S^2(\Omega^1_X)^{\vee\vee}$ directly in characteristic $0$ is not quite feasible with Magma~2.24-6\cite{magma}. Over a finite field, however, it can do so in only a matter of minutes, even for a $50$ digit characteristic for which $5$ is not a square in the prime field. We then use rational reconstruction to compute the trace and norm of each coefficient, and choose the conjugate in $\Q(\sqrt{5})$ that reduces to the coefficient. This allows us to lift the representations of the modules $M$, $M^\vee$, $M^{\vee\vee}$ together with the pairing matrices $A,B$ as in Section~\ref{S:comp_reg_diff}.

In order to verify that these reconstructed modules indeed have the right properties, we check that the matrices $A,B$ define well-defined pairings, i.e., that $AK_{M^\vee}\subset K_M$ etc.
This establishes that $M^{\vee\vee}_0$ indeed encodes sections of $S^2\Omega^1_X$ that are regular outside of the locus where an appropriate $3\times 3$ submatrix $A'$ is singular. Thus, if we establish that the base locus of the appropriate $3\times 3$ minors of $A$ is supported on the singular locus of $X$, then we establish that $M^{\vee\vee}_0$ determines global sections of $\hat{S}^2 \Omega^1_X$. Magma is capable of directly verifying in characteristic $0$ that the reconstructed module $M^{\vee\vee}$ equals its double dual, establishing that it is reflexive. See \cite{electronic} for a transcript of the computations verifying these claims.
We find that $h^0(X_6,\hat{S}^2\Omega^1(-H))=3$ and $h^0(X_6,\hat{S}^2\Omega^1)=15$. The forms themselves are a little unappetizing to display here. 

Let $\omega_1,\omega_2,\omega_3$ span $H^0(X,\hat{S}^2\Omega^1(-H))$. We apply Corollary~\ref{C:explicit_resultant_locus} to get information on genus $0$ curves $L$ for which $L\cap S$ is contained in a hyperplane and on genus $1$ curves $C$ for which $C\cap S$ spans a space of codimension at most one in the linear space spanned by $C$.  We find that outside of the vanishing locus of some of the $3\times 3$ minors $\det(A')$ of $A$, the locus
$\res(\omega_1,\omega_2)\cap \res(\omega_1,\omega_3)\cap \res(\omega_2,\omega_3)$ is $0$-dimensional. That shows that any such genus $0$ curve $L$ needs to lie in a plane spanned by singularities from $S$, or in the locus defined by $\det(A')=0$. 
This leaves us with analyzing finitely many loci. We can use the automorphism group of $X$ to significantly reduce the amount of computation required. We find the curves listed. See \cite{electronic} for a transcript of computations.
Interestingly, the loci $\det(A')=0$ also yield some non-planar genus $1$ curves passing through singularities that span $\PP^3$.
\end{proof}

\begin{remark} We mention here that in the masters thesis \cite{alaei2015} completed under the supervision of the first author, the same $27$ genus $0$ curves are already mentioned, and a similar argument to the one here is used to prove the slightly weaker result that any genus $0$ curve on $X_6$ has to pass throught at least one node.
\end{remark}


\appendix
\section{Hirzebruch-Riemann-Roch for twists of symmetric powers}
\label{sec:Chern}

We record the result of standard calculations of Chern classes of certain sheaves needed in the body of the article.

Let $Y$ be a smooth projective surface over a field. By a \defi{vector sheaf of rank} $r$ we mean a locally free sheaf of rank $r$ on $Y$; when $r = 1$ we call such a sheaf a \defi{line bundle}. Let $\calE$ be a vector sheaf of rank $2$ on $Y$.  Let $c_1(\calE)$, $c_2(\calE)$ be the usual Chern classes of $\calE$.  Using the splitting principle~\cite[Remark~3.2.3]{Fulton}, we compute the Chern classes for the symmetric power $\calA=S^m\calE$, a rank $m+1$ locally free sheaf on $Y$:
\begin{equation}
\label{eq:ChernClassesSm}
\begin{split}
	c_1(\calA)=c_1(S^m\calE) 	&= \tbinom{m+1}{2}c_1(\calE), \text{ and}\\
	c_2(\calA)=c_2(S^m\calE)	&= \tfrac{1}{24}(3m + 2)(m + 1)m(m-1)c_1^2(\calE) + \tfrac{1}{6}m(m+1)(m+2)c_2(\calE).
\end{split}
\end{equation}
For any vector sheaf $\calA$ of rank $(m+1)$ and a line bundle $\calL$ on $Y$ we have
\begin{equation}
\label{eq:ChernClassesProd}
\begin{split}
	c_1(\calA\otimes\calL) 	&= c_1(\calA) + (m+1)c_1(\calL), \\
	c_2(\calA\otimes\calL) 	&= c_2(\calA) + m c_1(\calA)c_1(\calL) + \tbinom{m+1}{2}c_1^2(\calL);
\end{split}
\end{equation}
see~\cite[p.~55]{Fulton}. 

Let $\calFA$ be a vector sheaf of rank $r$ on $Y$. Writing $K = -c_1(T_Y)$ and $\chi = c_2(T_Y)$ for the Chern classes of the tangent bundle $T_Y$ of $Y$, the Hirzebruch--Riemann--Roch Theorem  gives the Euler characteristic of $\calFA$ in terms of Chern classes of $\calFA$ and $T_Y$ (see~\cite[Example~15.2.2]{Fulton}):
\begin{equation}
\label{eq:RR}
\begin{split}
	\chi(Y,\calFA) &= \left(-\frac{1}{2}c_1(\calFA) K + \frac{1}{2}c_1^2(\calFA) - c_2(\calFA) + \frac{r}{12}(K^2 + \chi) \right)[Y].
\end{split}
\end{equation}
Here $[Y]$ is the fundamental class of $Y$. Together with~\eqref{eq:ChernClassesSm} and~\eqref{eq:ChernClassesProd}, Hirzebruch--Riemann--Roch affords the Euler characteristic of $\calFA := S^m\calE\otimes \calL$ in terms of $c_1(\calE)$, $c_2(\calE)$, $c_1(\calL)$, $K$ and $\chi$:
\begin{equation}
\label{eq:RRSmEL}
\begin{split}
\chi(Y,S^m \calE\otimes \calL)&=\tfrac{1}{6}\big(
c_1(\calE)^2-c_2(\calE)
\big)m^3\\
&\hspace{1em}-\tfrac{1}{4}\big(
c_1(\calE)K-c_1(\calE)^2-2c_1(\calE)c_1(\calL)+2c_2(\calE)
\big)m^2\\
&\hspace{1em}+\tfrac{1}{12}\big(
K^2-3c_1(\calE)K+c_1(\calE)^2-6c_1(\calL)K\\
&\hspace{4em}+6c_1(\calE)c_1(\calL)+6c_1(\calL)^2-4c_2(\calE)+\chi
\big)m\\
&\hspace{1em}+\tfrac{1}{12}\big(K^2-6c_1(\calL)K+6c_1(\calL)^2+\chi
\big).
\end{split}
\end{equation}
We specialize this result in two different ways.  First, setting $\calE = \Omega^1_Y$ and $\calL = \calO_Y$ and using $c_1(\Omega^1_Y)=K$ and $c_2(\Omega^1_Y)=\chi$, we get
\begin{equation}\label{E:chiYSm}
\chi(Y,S^m\Omega^1_Y)=\tfrac{1}{12}\big(
2(K^2-\chi)m^3-6\chi m^2-(K^2+3\chi)m+K^2+\chi
\big).
\end{equation}
Second, if $E$ is an irreducible ($-2$)-curve on $Y$, then
for $\calE=\Omega^1_Y(\log(E))$ and $\calL_h=-hE$ we have $c_1(\calE)=K+E$, $c_2(\calE)=\chi-2$, and $c_1(\calL_h)=-hE$, as well as $EK=0$, leading to
\begin{equation}\label{E:chirel}
\chi(Y,S^m\calE\otimes\calL_h)=\chi(Y,S^m\Omega^1_Y)-(m+1)(h^2+hm-\tfrac{1}{2}m).
\end{equation}


\end{document}